\documentclass[18pt,reqno,oneside]{amsart}
\usepackage{amsfonts,amssymb,latexsym,xspace,epsfig,graphics,color}
\usepackage{amsmath,enumerate,stmaryrd,xy}
\usepackage{natbib}
\usepackage{subfigure}

\newcommand{\Prob}{\mathbb{P}}

\newcommand{\XU}{[X({\bf U})]}

\newcommand{\G}{\ensuremath{\left}}    
\newcommand{\D}{\ensuremath{\right}}    
     
\newcommand{\Z}{\ensuremath{\mathbb{Z}}}    
   
\usepackage{algorithmic}
\usepackage{algorithm}

\newtheorem{lemma}{Lemma}[section]

\newtheorem{theo}{Theorem}[section]
\newtheorem{prop}{Proposition}[section]
\newtheorem{coro}{Corollary}[section]

\begin{document}
\DeclareGraphicsExtensions{.pdf,.gif,.jpg}

%%%%%%%%%%%%%%%%%%%%%%%%%%%%%%%%%%%%%%%%%%%%%%%%%%%%%%%%%%%%%%%%%%%%%
%\subjclass[2000]{60G10 (primary), 60G99 (secondary).}

%%%%%%%%%%%%%%%%%%%%%%%%%%%%%%%%%%%%%%%%%%%%%%%%%%%%%%%%%%%%%%%%%

\title[Perfect simulation for chains with infinite memory]{Perfect simulation for stochastic chains of infinite memory: relaxing the continuity assumption}
%\title[Perfect simulation for chains with infinite memory]{Processes with long memory: regenerative construction and perfect simulation}

%\author{A.~Gallo}
\author{Sandro Gallo \& Nancy L. Garcia}
%  \address{Universidade de S\~ao Paulo\\ Instituto de Matem\'atica e estat\'istica\\ 1010 Rua do Mat\~ao\\ BP 66281
%CEP 
%05315-970
%   S\~ao Paulo\\ Brasil} 
%  \email{gsandro@ime.usp.br}

%\date{February 24, 2009} 

\thanks{SG is supported by a FAPESP fellowship (grant 2009/09809-1),
  NG is supported by CNPq grants 475504/2008-9 and 301530/2007-6.}

\begin{abstract}
  This paper is composed of two main results concerning  chains of
  infinite order which are not necessarily continuous. The first one is a decomposition of the transition probability
  kernel as a countable mixture of unbounded probabilistic context
  trees. This decomposition is used to design a simulation algorithm
  which works as a combination of the algorithms given by
  \cite{comets/fernandez/ferrari/2002} and \cite{gallo/2009}.  The
  second main result gives sufficient conditions on the kernel for this algorithm
  to stop after an almost surely finite number of steps. Direct consequences of this last result are existence and uniqueness of the stationary chain compatible with the kernel.
\end{abstract}

\maketitle

\section{Introduction}\label{introduction}

The goal of this paper is to construct a perfect simulation scheme for
chains of infinite order on a countable alphabet, compatible with
a transition probability kernel which is not necessarily
continuous. By a perfect simulation algorithm we mean an algorithm
which samples precisely from the stationary law of the process.

Perfect simulation for chains of infinite order was first done by
\cite{comets/fernandez/ferrari/2002} under the continuity
assumption. They used the fact (observed earlier by
\cite{kalikow/1990}) that under this assumption, the transition
probability kernel can be decomposed as a countable mixture of Markov
kernels. Then, \cite{gallo/2009} obtained a perfect simulation algorithm for
chains compatible with a class of unbounded probabilistic context
trees where each infinite size branch can be a discontinuity point.

In this paper, we consider a class of transition probability
kernels which are neither necessarily continuous nor necessarily
probabilistic context trees. In fact, the same infinite size branches
as in the context trees considered by \cite{gallo/2009} are allowed to be discontinuity points, and the
other branches must have a certain localized-continuity
assumption. Under these new assumptions, we obtain a Kalikow-type
decomposition of our kernels as a mixture of unbounded probabilistic context trees. The fact that our decomposition involves unbounded probabilistic context trees instead of Markov kernels (as it was the case for \cite{kalikow/1990}) seems to be ``the price to pay'' to allow discontinuities at some points.

As a consequence of this decomposition and some
minimum extra condition, we can
show that there exists at least one stationary chain compatible with our kernels,
extending the existing result stating that continuity was
sufficient. A perfect simulation is then constructed using this
decomposition together with the \emph{coupling from the past} (CFTP) method introduced in the seminal paper of \cite{propp/wilson/1996}.
One of the main consequence of the existence of a perfect simulation algorithm is the fact that there exists a unique stationary chain compatible with our kernels.

More precise explanations of what is done here are postponed to Section \ref{sec:motivation} since we need the notation and definitions given in Section 2. Our first main result, Theorem \ref{theo1} which is stated and
proved in Section 4, gives the decomposition which holds without the continuity condition.  In Section 5,
we explain how our perfect simulation works using Theorem
\ref{theo1}, and we present it under the form of the \emph{pseudo-code}, Algorithm 1.  After that, we state our second main theorem, Theorem
\ref{theo2} which says that Algorithm 1 stops almost surely after a
finite number of steps. Section \ref{laprova} is dedicated to the
proof of Theorem \ref{theo2}.  We finish this paper with some comments
and further questions.

\section{Notation and definitions}\label{Notations, definitions...}
Let $A$ be a countable alphabet. Given two integers $m\leq
n$, we denote by $a_m^n$ the string $a_m \ldots a_n$ of symbols in
$A$. For any $m\leq n$, the length of the string $a_m^n$ is denoted by
$|a_m^n|$ and is defined by $|a_m^n| = n-m+1$. For any
$n\in\mathbb{Z}$, we will use the convention that
$a_{n+1}^{n}=\emptyset$, and naturally $|a_{n+1}^{n}|=0$. Given two
strings $v$ and $v'$, we denote by $vv'$ the string of length $|v| +
|v'| $ obtained by concatenating the two strings. The concatenation of
strings is also extended to the case where $v$ denotes a semi-infinite
sequence, that is $v=\ldots a_{-2}a_{-1}$, $a_{-i}\in A$ for $i\geq1$.
If $n$ is a positive integer and $v$ a finite string of symbols in
$A$, we denote by $v^{n}=v\ldots v$ the concatenation of $n$ times
the string $v$.  We denote
$$
A^{-\mathbb{N}}=A^{\{\ldots,-2,-1\}}\,\,\,\,\,\,\textrm{ and }\,\,\,\,\,\,\,  A^{\star} \,=\, \bigcup_{j=0}^{+\infty}\,A^{\{-j,\dots, -1\}}\, ,
$$
which are, respectively, the set of all infinite strings of past symbols and the set of all finite strings of past symbols. 
The case $j=0$ corresponds to the empty string $\emptyset$. Finally, we denote by    $\underline{a}=\ldots a_{-2}a_{-1}$ the elements of $A^{-\mathbb{N}}$.

%\subsection{Family of transition probabilities and chains of infinite memory}

%This brief subsection aims to remember some basic definitions of the theory.

%\begin{defi}
%A \emph{family of probability transitions} is a function $P:A\times A^{-\mathbb{N}}\rightarrow [0,1]$ such that for any $\underline{a}\in A^{-\mathbb{N}}$
%\[
%\sum_{a\in A}P(a|\underline{a})=1.
%\]
%\end{defi}

%
%\begin{defi}
%A stationary chain $(X_{n})_{n\in\mathbb{Z}}$ of law $\mathbb{P}$ is compatible with a system of probability transitions $P$ if for any past $\underline{a}\in A^{-\mathbb{N}}$ and any $a\in A$ we have
%\begin{equation}\label{compatible}
%\mathbb{P}(X_{0}=a|X_{-\infty}^{-1}=\underline{a})=P(a|\underline{a}).
%\end{equation}
%\end{defi}
%These chains are 
%\begin{itemize}
%\item Bernoulli chains if $P(a|\underline{a})$ does not depend on $\underline{a}$,
%\item $k$-step Markov chains if $P(a|\underline{a})=P(a|a_{-k}^{-1})$ independently of $a_{-\infty}^{-k-1}$ for some finite $k$,
%\item of infinite memory otherwise.
%\end{itemize}

\subsection{Standard definitions}

A transition probability kernel (or simply \emph{kernel} in the
sequel) on an alphabet $A$ is a function
\begin{equation}
\begin{array}{cccc}
P:&A\times A^{-\mathbb{N}}&\rightarrow& [0,1]\\
&(a,\underline{a})&\mapsto&P(a|\underline{a})
\end{array}
\end{equation}
such that
\[
\sum_{a\in A}P(a|\underline{a})=1\,\,,\,\,\,\,\,\,\forall \underline{a}\in A^{-\mathbb{N}}.
\]
%In the present work, we only consider non-markovian family of transition probabilities. 
In this paper, we consider kernels $P$ which depends on an unbounded
part of the past, unlike the markovian case.  A stationary stochastic
chain ${\bf X}=(X_{n})_{n\in\Z}$ on $A$ having law $\mu$ is said to be
\emph{compatible} with a kernel $P$ if the later is a regular version
of the conditional probabilities of the former, that is
\begin{equation}\label{compa}
\mu(X_{0}=a|X_{-\infty}^{-1}=\underline{a})=P(a|\underline{a})
\end{equation}
for every $a\in A$ and $\mu$-almost every $\underline{a}$ in $A^{-\mathbb{N}}$. We  call these chains \emph{chains of infinite memory}.

\subsection{Probabilistic context tree}\label{pct}

We say that a kernel $P$ has a \emph{probabilistic context tree} representation if there exists a  function $d:A^{-\mathbb{N}}\rightarrow \mathbb{N}\cup\{+\infty\}$ such that for any two infinite sequences of past symbols $\underline{a}$ and $\underline{b}$
\[
a_{-d(\underline{a})}^{-1}=b_{-d(\underline{a})}^{-1}\Rightarrow P(a|\underline{a})=P(a|\underline{b}).
\]
It follows that the length $d(\underline{a})$ only depends on the suffix $a_{-d(\underline{a})}^{-1}$ of $\underline{a}$. This allows us to identify the set $\tau:=\{a_{-d(\underline{a})}^{-1}\}_{\underline{a}\in A^{-\mathbb{N}}}$  with  the set of leaves of a rooted tree where each node has either $|A|$ sons (internal node) or $0$ sons (leaf). The set $\tau$ is called the  \emph{context tree}, ``context'' being the original name \cite{rissanen/1983} gave to the strings
\[
c_{\tau}(\underline{a}):=a_{-d(\underline{a})}^{-1}
\]
when he introduced this model. A probabilistic context tree is an ordered pair $(\tau,p)$ where $\tau$ is a context tree and $p:=\{p(a|v)\}_{a\in A,v\in\tau}$ is a set of transition probabilities associated to each element of $\tau$. Thus, the probabilistic context tree $(\tau,p)$ \emph{represents} the kernel $P$ if for any $\underline{a}\in A^{-\mathbb{N}}$ and any $a\in A$
\[
P(a|\underline{a})=p(a|c_{\tau}(\underline{a})).
\]

Examples of probabilistic context trees are shown in Figures \ref{fig:finite-pct} (for the bounded case) and  \ref{fig:infinite-pct} (for the unbounded case). In the first one, at each leaf (context) of the tree we associate three boxes in which are given   the transition probabilities to each symbols of $A$ given this context. In the second one, we only specify the probability $p_{i}:=p(2|0^{i}2)$ (observe that we swap the order when we write the context in a conditioning), the transition probabilities to $1$ are simply  $1-p_{i}$.
%More general examples of unbounded context trees (without specifying the transition probabilities) are given by Figures \ref{fig:partition} and \ref{fig:partition2}.
\begin{figure}[h!]
\begin{center}
\subfigure[]{
\includegraphics[scale=0.9]{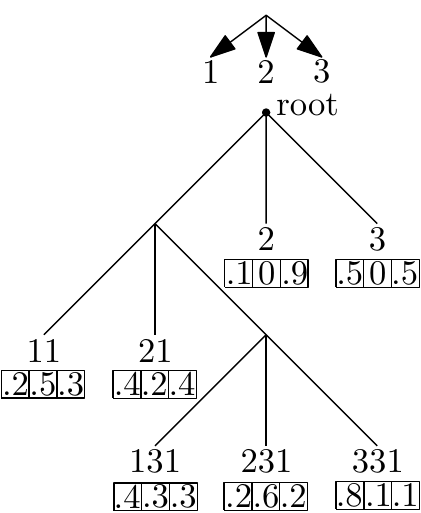}
\label{fig:finite-pct}}
\hspace{1cm}
\subfigure[]{
\includegraphics[scale=1]{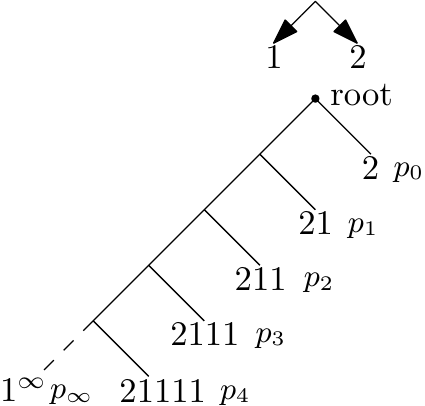}
\label{fig:infinite-pct}}
\caption{Examples of probabilistic context trees.}
\end{center}
\end{figure}

A stochastic chain ${\bf X}$ compatible, in the sense of
(\ref{compa}), with a probabilistic context tree is called a
\emph{chain with variable length memory}.

\section{Motivation}\label{sec:motivation}

The aim of this section is twofold: it motivates and explains at the same time, the present work.
%\subsection{Under the continuity assumption: \cite{comets/fernandez/ferrari/2002}}

\subsection{Countable mixture of Markov kernels  under the continuity assumption}

We say that a point (an infinite sequence of past symbols)
$\underline{a}$ is a continuity point for a given transition
probabilities kernel $P$ if
\[
\beta_{k}(\underline{a}):=\sup_{a\in A}\sup_{\underline{y},\underline{z}}|P(a|a_{-k}^{-1}\underline{y})-P(a|a_{-k}^{-1}\underline{z})|\stackrel{k\rightarrow+\infty}{\longrightarrow}0.
\]
%where we use the convention $a_{0}^{-1}=\emptyset$.  
Otherwise, we say that $\underline{a}$ is a discontinuity point for
$P$. $P$ is said to be continuous if
\begin{equation}
\label{eq:cont}
\beta_{k}:=\sup_{a_{-k}^{-1}}\beta_{k}(\underline{a})\stackrel{k\rightarrow+\infty}{\longrightarrow}0.
\end{equation}

% A more general (perhaps too much general) assumption, should be to
% consider the case where
%\[
%\beta_{k}(\underline{a}):=\sup_{a\in A}\sup_{\underline{y},\underline{z}}|P(a|a_{-k}^{-1}\underline{y})-P(a|a_{-k}^{-1}\underline{z})|
%\]
%goes to zero for a \emph{large set} (of measure one) of points $\underline{a}$.
\cite{kalikow/1990} showed that  continuous transition
probability kernels $P$ can be represented through the form of a countable
mixture of Markov kernels, that is, there exist two probability
distributions $\{p^{CFF}_{0}(a)\}_{a\in A}$ and
$\{\lambda^{CFF}_{k}\}_{k\geq0}$ and a sequence of Markov kernels
$\{p^{CFF}_{k}\}_{k\geq1}$ such that for any $a\in A$ and
$\underline{z}\in A^{-\mathbb{N}}$
\begin{equation}\label{odeles}
P(a|\underline{z})=\lambda^{CFF}_{0}p^{CFF}_{0}(a)+\sum_{k\geq1}\lambda^{CFF}_{k}p^{CFF}_{k}(a|z_{-k}^{-1}).
\end{equation} 
The superscript ``CFF'' refers to the fact that we will use the
definitions from \cite{comets/fernandez/ferrari/2002}.

Define an
$\mathbb{N}$-valued random variable $K^{CFF}$ taking value $k\geq0$
w.p. $\lambda^{CFF}_{k}$.  Decomposition (\ref{odeles}) means the
following. To choose the next symbol looking at the whole past
$\underline{z}$ using the distribution $\{P(a|\underline{z})\}_{a\in
  A}$ is equivalent to the following two steps procedure:
\begin{enumerate}
\item choose $K^{CFF}$,
\item 
\begin{enumerate}
\item if $K^{CFF}=0$, then choose the next symbol w.p. $\{p_{0}(a)\}_{a\in A}$,
\item if $K^{CFF}=k>0$ then choose the next symbol looking at $z_{-k}^{-1}$  and using $\{p^{CFF}_{k}(a|z_{-k}^{-1})\}_{a\in  A}$.
\end{enumerate}
\end{enumerate}
Observe that $K^{CFF}$ is independent of everything (in particular, it does not depend on $\underline{z}$).
%At each time $i$ there exists a random variable $k^{CFF}_{i}$ (independent of $k^{CFF}_{j}$ for $j\neq i$) taking value $k\in\mathbb{N}$ with probability $\lambda^{CFF}_{k}$, such that the state of the chain at time $i$ depends only on the last $k^{CFF}_{i}$ values of the chain. In the case where $k^{CFF}_{i}=0$, the state of the chain is decided independently of the rest, with probability $p^{CFF}_{0}(\cdot)$. In other words, suppose we are given an infinite past $\underline{z}$ and that $X_{-\infty}^{i-1}=\underline{z}$. Then, to construct the next value $X_{i}$ of the chain using $P$ is equivalent to pick up a random length $k_{i}$, and to decide $X_{i}$ independently of $X_{-\infty}^{i-k_{i}-1}$.

To clarify the parallel between this decomposition and the
decomposition presented in Theorem \ref{theo1}, we explain how
\cite{comets/fernandez/ferrari/2002} define the distribution
$\{\lambda^{CFF}_{k}\}_{k\geq0}$. For any $a\in A$ and $a_{-k}^{-1}\in
A^{k}$ they consider the functions
\[
\alpha^{CFF}_{0}(a)=\inf_{\underline{z}}P(a|\underline{z})\,\,\textrm{ and }\,\,
\alpha^{CFF}_{k}(a|a_{-k}^{-1})=\inf_{\underline{z}}P(a|a_{-k}^{-1}\underline{z})
\]
and the sequence $\{\alpha^{CFF}_{k}\}_{k\geq0}$
defined by $\alpha^{CFF}_{0}=\sum_{a\in A}\alpha^{CFF}_{0}(a)$ and for
any $k\geq1$
\[
\alpha^{CFF}_{k}=\inf_{a_{-k}^{-1}\in A^{w}}\sum_{a\in A}\alpha^{CFF}_{k}(a|a_{-k}^{-1}).
\]
These are, as they say, ``probabilistic threshold for memories limited
to k preceding instants.'' Taking the infimum over every $a_{-k}^{-1}$
is related to the continuity assumption (\ref{eq:cont}). In fact, to
assume continuity is equivalent to assume that $\alpha^{CFF}_{k}$ goes
to $1$ as $k$ diverges, and to assume punctual continuity in
$\underline{a}$ is equivalent to assume that
\[
\alpha_{k}^{CFF}(\underline{a}):=\sum_{a\in A}\alpha^{CFF}_{k}(a|a_{-k}^{-1})
\]
goes to $1$ as $k$ diverges.  Under the continuity assumption
(\ref{eq:cont}), the probability distribution
$\{\lambda^{CFF}_{k}\}_{k\geq0}$ used in (\ref{odeles}) is defined as
follows: $\lambda^{CFF}_{k}=\alpha^{CFF}_{k}-\alpha^{CFF}_{k-1}$ for
$k\geq1$ and $\lambda^{CFF}_{0}=\alpha^{CFF}_{0}$.

\subsection{Without the continuity assumption}

%To assume that $\beta_{k}(\underline{a})$ is a discontinuity point is equivalent to assume that $\sum_{a\in A}\alpha^{CFF}_{k}(a|a_{-k}^{-1})$ does not converge to $1$.
To fix ideas, in the remaining of this section, assume $P$ is a
transition probability kernel on $A=\{1,2\}$ which has a single
discontinuity point which is $1^{-\mathbb{N}}$. Then
$\alpha^{CFF}_{k}(\underline{a})$ goes to $1$ as $k$ diverges if and
only if $\underline{a}\neq1^{-\mathbb{N}}$.  In this case,
$\alpha^{CFF}_{k}$ does not converge to $1$ and the above result does
not apply. 

\subsubsection{The context tree assumption}\label{hahaaa}

\cite{gallo/2009}   assumed that $P$ is represented by the probabilistic context tree  $(\tau,p)$, where
\[
\tau=1^{-\mathbb{N}}\cup\bigcup_{i\geq0}\bigcup_{a^{-1}_{-\ell(i)}\in A^{\ell(i)}}a^{-1}_{-\ell(i)}\,2\,1^{i},
\]
$\ell:\mathbb{N}\rightarrow\mathbb{N}$ being  a deterministic function. This context tree is represented in Figure \ref{fig:thepartition}.
%This is what we call the context tree assumption, and the corresponding context tree is illustrated Figure \ref{fig:thepartition}. Let us call $\tau$ this context tree. Let us denote for any $i\geq0$ and $k>\ell^{2}(i)$
%\[
%P(a|1^{i}2a_{-k}^{-1}\underline{z})=p(a|1^{i}2a_{-\ell^{2}(i)}^{-1}).
%\]
%Therefore, our family $P$ is such that for any $a\in A$ and $\underline{z}\in A^{-\mathbb{N}}$
%\[
%P(a|\underline{z})=p(a|c_{\tau}(\underline{z})).
%\]
\begin{figure}
\centering
\includegraphics[scale=0.9]{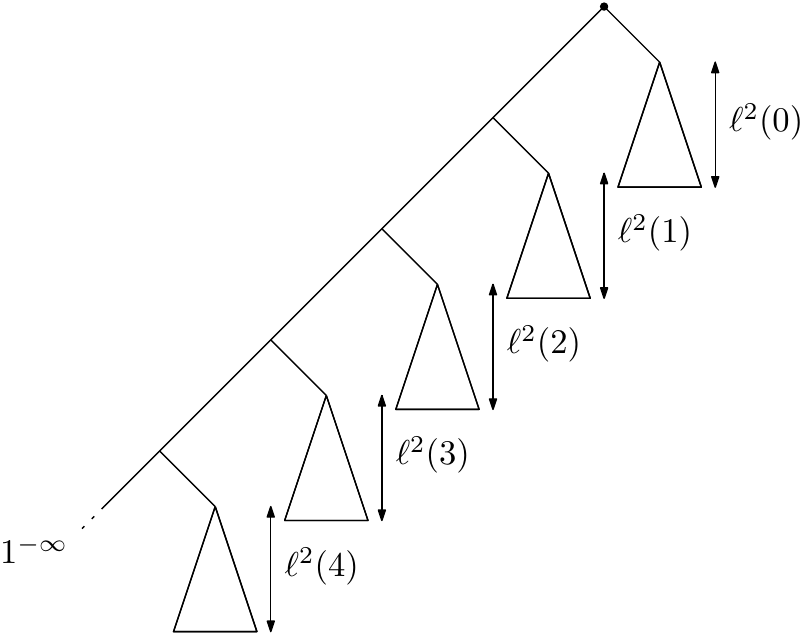}
\caption{}
\label{fig:thepartition}
\end{figure}
Observe that, under this assumption, for any $i\geq0$ and
$k>\ell(i)$ we have
\begin{equation}\label{contexttreeassumption}
P(a|1^{i}2a_{-k}^{-1}\underline{z})=P(a|1^{i}2a_{-k}^{-1}\underline{y}),
\quad \mbox{for any } \underline{z} \mbox{ and } \underline{y} \in
A^{-\mathbb{N}}.
\end{equation}
It follows that
\[
\inf_{a_{-k}^{-1}\in A^{k}}\sum_{a\in A}\alpha^{CFF}_{k+i+1}(a|1^{i}2a_{-k}^{-1})=1
\]
whenever $k>\ell(i)$. Therefore $\sum_{a\in A}\alpha^{CFF}_{k}(a|a_{-k}^{-1})$ goes to $1$ as $k$ diverges for any $\underline{a}\neq1^{-\mathbb{N}}$. We do not specify what happens for the point $\underline{a}=1^{-\mathbb{N}}$. Making a parallel with the above case, we can decompose such a kernel as follows: for any $a\in A$ and $\underline{z}$
\[
P(a|\underline{z})=\lambda^{CFF}_{0}p^{CFF}_{0}(a)+(1-\lambda^{CFF}_{0})p'(a|c_{\tau}(\underline{z}))
\]
where
\[
p'(a|c_{\tau}(\underline{z})):=\frac{p(a|c_{\tau}(\underline{z}))-\lambda_{0}^{CFF}p_{0}^{CFF}(a)}{1-\lambda_{0}^{CFF}}.
\]
Define an $\mathbb{N}$-valued random variable $K^{G}$ which takes value  $0$  w.p. $\lambda^{CFF}_{0}$ or $|c_{\tau}(\underline{z})|$ w.p. $1-\lambda^{CFF}_{0}$.
The context tree assumption for $P$ means the following. To choose the next symbol looking at the whole past $\underline{z}$ using the distribution $\{P(a|\underline{z})\}_{a\in A}$ is equivalent to the following two steps procedure:
\begin{enumerate}
\item choose $K^{G}$,
\item 
\begin{enumerate}
\item if $K^{G}=0$, choose the next symbol w.p. $\{p^{CFF}_{0}(a)\}_{a\in A}$,
\item if $K^{G}=|c_{\tau}(\underline{z})|$,   choose the next symbol looking at $c_{\tau}(\underline{z})$  and using $\{p'_{k}(a|c_{\tau}(\underline{z})\}_{a\in  A}$.
\end{enumerate}
\end{enumerate}
Observe that the random variable $K^{G}$ is a deterministic function of the past $\underline{z}$ whenever its value is not $0$: $K^{G}=|c_{\tau}(\underline{z})|$.
%In this case, there exists a random variable $k_{i}^{G}$ (the superscript ``G'' refer to \cite{gallo/2009})  associated to each time index $i$ which tells us the length of the dependence at time $i$. It takes value $0$ w.p. $\lambda_{0}^{CFF}$, and it takes value (recall the definition of $m^{w}$ in the present case where $w=2$) $k^{G}_{i}(\underline{z})=m^{2}(\underline{z})+1+\ell^{2}(m^{2}(\underline{z}))$ with probability $1-\alpha_{0}^{CFF}$. Whenever $k^{G}_{i}(\underline{z})\neq0$, it is a deterministic function of $\underline{z}$.
%In other words, the length of the dependence to the past is a deterministic function of the past $\underline{z}$ itself. 

\subsubsection{The countable mixture of probabilistic context trees}\label{altern}

So far, two extreme cases have been considered: $K^{G}$ is a
deterministic function of the past, and $K^{CFF}$ is a random variable
totally independent of the past. In the present work, we introduce a way to
combine these two approaches.  It allows us to consider kernels $P$
which are neither necessarily represented by a probabilistic context tree,
nor necessarily continuous. This new approach is based on the
assumption that
%, However, if we define for $k\geq1$ and $i\geq0$
%\[
%\alpha'_{k}(a|1^{i}2a_{-k}^{-1}):=\inf_{\underline{z}}P(a|1^{i}2a_{-k}^{-1}\underline{z})
%\]
%with the convention that $a_{0}^{-1}=\emptyset$ and then
\begin{equation}\label{assum}
\alpha_{k}:=\inf_{i\geq0}\inf_{a_{-k}^{-1}\in A^{k}}\sum_{a\in
  A}\alpha^{CFF}_{k+i+1}(a|1^{i}2a_{-k}^{-1}) \quad \stackrel{k
  \rightarrow \infty}{\longrightarrow} \quad 1.
\end{equation}
The $\alpha_{k}$'s are ``probabilistic threshold for memories going
until the $k^{\textrm{th}}$  instant preceding the last occurrence of
symbol $2$ in the past.'' In this case also, we have that $\sum_{a\in
  A}\alpha^{CFF}_{k}(a|a_{-k}^{-1})$ goes to $1$ as $k$ diverges for
any $\underline{a}\neq1^{-\mathbb{N}}$ and not necessarily for
$1^{-\mathbb{N}}$. Notice that the probabilistic context tree
assumption we introduced in Section \ref{hahaaa} only satisfies 
\[
\inf_{a_{-k}^{-1}\in A^{k}}\sum_{a\in
  A}\alpha^{CFF}_{k+i+1}(a|1^{i}2a_{-k}^{-1}) \quad \stackrel{k
  \rightarrow \infty}{\longrightarrow} \quad 1.\]
 Under assumption (\ref{assum}), it will be shown in the next section that there exists a probability distribution $\{\lambda_{k}\}_{k\geq0}$, and a sequence of probabilistic context trees $\{(\tau_{k},p_{k})\}_{k\geq0}$   such that 
\begin{equation}\label{mixture}
P(a|\underline{z})=\lambda^{CFF}_{0}p^{CFF}_{0}(a)+\sum_{k\geq0}\lambda_{k}p_{k}(a|c_{\tau_{k}}(\underline{z})).
\end{equation}
The  $k$th context tree of decomposition (\ref{mixture}) is given by
\begin{equation}\label{eq:3}
\tau_{k}:=1^{-\mathbb{N}} \,\cup\,\bigcup_{i\geq0} \bigcup_{a_{-k}^{-1}\in A^{k}}a_{-k}^{-1}\,2\,1^i.
\end{equation}
The sequence of context trees $\{\tau_{k}\}_{k\geq0}$ 
for the present particular case is illustrated in Figure
\ref{fig:sequencehu}. Define a random variable $K^{GG}$ taking
values $0$ w.p. $\lambda^{CFF}_{0}$ and
$|c_{\tau_{k}}(\underline{z})|$ w.p. $\lambda_{k}$ for $k\geq0$.  One
more time, let us translate this decomposition into a two steps
procedure:
\begin{enumerate}
\item choose $K^{GG}$,
\item 
\begin{enumerate}
\item if $K^{GG}=0$, choose the next symbol w.p. $\{p^{CFF}_{0}(a)\}_{a\in A}$,
\item if $K^{GG}=|c_{\tau_{k}}(\underline{z})|$ for some $k\geq0$,   choose the next symbol looking at $c_{\tau_{k}}(\underline{z})$  and using $\{p_{k}(a|c_{\tau_{k}}(\underline{z})\}_{a\in  A}$.
\end{enumerate}
\end{enumerate}
Observe that this time, the random variable $K^{GG}$ depends on the past $\underline{z}$, but through a random mechanism using the distribution $\{\lambda_{k}\}_{k\geq0}$.

\begin{figure}[htp]
\centering
\includegraphics{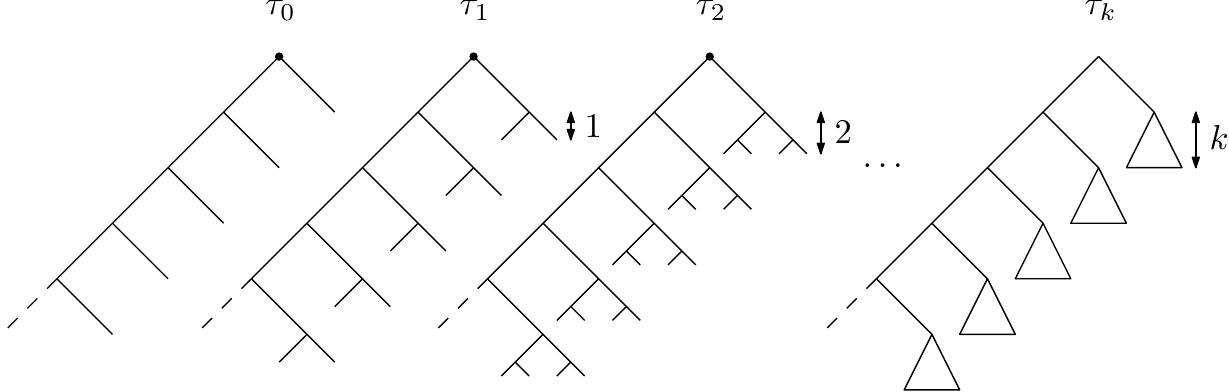}
\caption{}
\label{fig:sequencehu}
\end{figure}

In the next section, we state our first main result in a general
framework. The alphabet can be countable and the role which is played above
by symbol $2$ can be played by any finite string $w\in A^{\star}$. In
this case, we allow $P$ to have discontinuities at every point
$\underline{z}\in A^{-\mathbb{N}}$ which does not have $w$ as subsequence.

\section{First main result: a countable mixture of unbounded probabilistic context trees}\label{general}
%
%\subsection{The countable mixture representation of $P$}\label{formalize}
\subsection{Some more definitions and statement of the first results}
%We put
%\[
%\mathcal{E}:=\{a\in A:\inf_{\underline{z}}P(a|\underline{z})>0\}.
%\]
%Without loss of generality, we can write $A=\{1,2,3,\ldots\}$ and $\mathcal{E}=\{1,\ldots,\#\mathcal{E}\}$ where $\#\mathcal{E}$ denotes the cardinal of $\mathcal{E}$. 
Fix a finite size string $w\in A^{\star}$ and define the function $m^{w}$ which associates to any string  $a_{-m}^{-1}\in A^{m}$, $|w|\leq m\leq+\infty$ the distance to the first occurrence of $w$  when we look backward in $a_{-m},\ldots,a_{-1}$, that is 
\begin{equation}\label{lem}
m^{w}(a_{-m}^{-1})=\inf\{k\geq0:a_{-k-|w|}^{-k-1}=w\},
\end{equation}
we use the convention $m^{w}(a_{-m}^{-1})=+\infty$  if the set of indexes is empty.
Using this definition,  we introduce
\[
\mathcal{I}^{k}(\bar{w}):=\{(a_{-k},\ldots,a_{-1}):\underline{a}\in A^{-\mathbb{N}}\,\textrm{ and }\,m^{w}(\underline{a})=k\}
\] 
which is the set of strings $v$ of length $k$ such that there is a
unique occurrence of $w$ in the concatenation $wv$.
For $w,v\in A^{\star}$, $|w|\leq |v|$, we
use the abuse of notation $w\in v$ (\emph{resp.} $w\notin v$) which
means ``$w$ is (\emph{resp.} is not) a substring of $v$''. Then, for
any string $w\in A^{\star}$ and  $k\geq|w|$, we denote the set of the strings of
length $k$ in which $w$ does not appear as a subsequence by
\[
A^{k}(\bar{w}):=\{a_{-k}^{-1}\in A^{k}:a_{i}^{i+|w|-1}\neq
w,\,i=-k,\ldots,-|w|\}.
\]
Its complement is denoted by $A^{k}(w)=A^{k}\setminus
A^{k}(\bar{w})$. Finally, $A^{-\mathbb{N}}(\bar{w})$, denotes the set
of infinite sequences of past symbols $\underline{a}$ such that
$w\notin \underline{a}$, and
$A^{-\mathbb{N}}(w):=A^{-\mathbb{N}}\setminus
A^{-\mathbb{N}}(\bar{w})$.  Observe that 
$\mathcal{I}^{k}(\bar{w})$ can be different from $A^{k}(\bar{w})$.

\begin{theo}\label{theo1}
Consider a  transition probability kernel $P$ such that
\begin{equation}\label{condition1}
\alpha^{w}_{k}:=\inf_{i\geq0}\inf_{b_{-i}^{-1}\in \mathcal{I}^{i}(\bar{w})}\inf_{c_{-k}^{-1}\in A^{k}}\sum_{a\in A}\inf_{\underline{z}}P(a|b_{-i}^{-1}\,w\,c_{-k}^{-1}\,\underline{z})\stackrel{k\rightarrow+\infty}{\longrightarrow}1.
\end{equation}
Then, there exist two probability distributions $\{\lambda^{w}_{k}\}_{k\geq-1}$ and $\{p^{w}_{-1}(a)\}_{a\in A}$, and a sequence of probabilistic context trees $\{(\tau^{w}_{k},p^{w}_{k})\}_{k\geq0}$ such that 
\begin{equation}\label{mixture}
P(a|\underline{z})=\lambda^{w}_{-1}p^{w}_{-1}(a)+\sum_{k\geq0}\lambda^{w}_{k}p^{w}_{k}(a|c_{\tau^{w}_{k}}(\underline{z})).
\end{equation}
\end{theo}

%As we mention in the introduction, allowing discontinuities at any points of $A^{-\mathbb{N}}(\bar{w})$

\begin{coro}\label{coro1}
  Under the same condition of  Theorem \ref{theo1}, if
\begin{equation}
\label{eq:corol1}
\inf_{a\in w}\inf_{\underline{a}\in A^{-\mathbb{N}}}P(a|\underline{a})>0,
\end{equation}
then there exists at least one stationary chain compatible with $P$ in
the sense of (\ref{compa}).
\end{coro}

Corollary \ref{coro1} follows from Theorem \ref{theo1} by the same
arguments used in proof of Theorem 11 in
\cite{kalikow/1990}. The decomposition of $P$ as mixture of
probabilistic context trees together with assumption (\ref{eq:corol1})
provide the necessary features to obtain that the limit of the
empirical measures is in fact compatible with $P$.

Since the assumption of Theorem \ref{theo1} is not intuitive, let us give sufficient conditions for Theorem \ref{theo1} to hold.

\begin{prop}\label{prop}
\begin{equation}\label{one}
\inf_{a_{-k}^{-1}\in A^{k}}\sum_{a\in A}\inf_{\underline{z}}P(a|a_{-k}^{-1}\underline{z})\stackrel{k\rightarrow+\infty}{\longrightarrow}1
\end{equation}
implies that
\begin{equation}\label{two}
\inf_{a_{-k}^{-1}\in A^{k}(w)}\sum_{a\in A}\inf_{\underline{z}}P(a|a_{-k}^{-1}\underline{z})\stackrel{k\rightarrow+\infty}{\longrightarrow}1
\end{equation}
which implies that $\alpha^{w}_{k}$ converges to $1$ as $k$ diverges.
\end{prop}

A consequence of this proposition is that the assumption of Theorem
\ref{theo1} is weaker than the continuity assumption. It follows, in
particular, that the \emph{existence} result of Corollary \ref{coro1}
extends the well-known fact that existence hold whenever the 
transition probability kernel is continuous.

The proofs of Theorem \ref{theo1} and
Proposition \ref{prop} are given in Section \ref{sec:proof1}. 

Theorem \ref{theo1} is based on the existence of a triplet of
parameters (which is not unique): two probability distributions
$\{\lambda^{w}_{k}\}_{k\geq-1}$ and $\{p^{w}_{-1}(a)\}_{a\in A}$, and
a sequence of probabilistic context trees
$\{(\tau^{w}_{k},p^{w}_{k})\}_{k\geq0}$.  What follows is dedicated to
the definition of such a triplet of parameters. 

%$(\lambda^{w}_{k})_{k\geq-1}$ and $(p^{w}_{-1}(a))_{a\in A}$, and the sequence of probabilistic context trees $(\tau^{w}_{k},p^{w}_{k})$ for $k\geq0$. These are not unique, in fact

\subsection{A triplet of parameters
  $\left(\{\lambda^{w}_{k}\}_{k\geq-1},\,\{p^{w}_{-1}(a)\}_{a\in
      A},\,\{(\tau^{w}_{k},p^{w}_{k})\}_{k\geq0}\right)$}

We fix a string $w$ of $A^{\star}$. 
%To avoid overloaded notation,  we will omit the superscript $w$ in all the quantities depending on it, when no ambiguity is possible. 
The definition of our triplet is based on two partitions of $[0,1[$ inspired by \cite{comets/fernandez/ferrari/2002}.
Let us first show that for any $a\in A$, $i\geq0$, $b_{-i}^{-1}\in   \mathcal{I}^{i}(\bar{w})$ and $\underline{c}\in A^{-\mathbb{N}}$
\begin{equation}\label{ca}
\inf_{\underline{z}}P(a|b_{-i}^{-1}\,w\,c_{-k}^{-1}\underline{z})\stackrel{k\rightarrow+\infty}{\longrightarrow}P(a|b_{-i}^{-1}w\underline{c}).
\end{equation}
Observe that
\[
0\leq\sum_{a\in A}\G(P(a|b_{-i}^{-1}w\underline{c})-\inf_{\underline{z}}P(a|b_{-i}^{-1}\,w\,c_{-k}^{-1}\,\underline{z})\D)
=1-\sum_{a\in A}\inf_{\underline{z}}P(a|b_{-i}^{-1}wc_{-k}^{ -1}\underline{z}).
\]
Moreover,
\[
\sum_{a\in A}\inf_{\underline{z}}P(a|b_{-i}^{-1}wc_{-k}^{ -1}\underline{z})\geq \inf_{i\geq0}\inf_{b_{-i}^{-1}\in \mathcal{I}^{i}(\bar{w})}\inf_{c_{-k}^{-1}\in A^{k}}\sum_{a\in A}\inf_{\underline{z}}P(a|b_{-i}^{-1}\,w\,c_{-k}^{-1}\,\underline{z})
\]
therefore,  under assumption (\ref{condition1}), $\sum_{a\in A}\inf_{\underline{z}}P(a|b_{-i}^{-1}wc_{-k}^{ -1}\underline{z})$ goes to $1$ and
%\[
%\inf_{a_{-(k+i+|w|)}^{-1}\in A_{w}^{k+i+|w|}}\sum_{a\in A}\inf_{\underline{z}}P(a|a_{-(k+i+|w|)}^{-1}\underline{z}),
%\]
%therefore
\[
\sum_{a\in A}\G(P(a|b_{-i}^{-1}\,w\,\underline{c})-\inf_{\underline{z}}P(a|b_{-i}^{-1}\,w\,c_{-k}^{-1}\,\underline{z})\D)\stackrel{k\rightarrow+\infty}{\longrightarrow}0.
\]
%\[
%1-\inf_{a_{-(k+i+|w|)}^{-1}\in A_{w}^{k+i+|w|}}\sum_{a\in A}\inf_{\underline{z}}P(a|a_{-(k+i+|w|)}^{-1}\underline{z})
%\]
Since all the terms in the sum over $A$ are positive the convergence of (\ref{ca}) holds.
\subsubsection{Definition of the first partition of [0,1[}

We will denote for any $v\in A^{-\mathbb{N}}\cup A^{\star}$ such that
     $m^{w}(v)<+\infty$
\[
\alpha^{w}(a,v,k):=\inf_{\underline{z}}P\left(a\Big|v_{-m^{w}(v)}^{-1}\,w\,v_{-(m^{w}(v)+|w|+k)}^{-(m^{w}(v)+|w|+1)}\,\underline{z}\right)
\]
for $k=0,\ldots,|v|-m^{w}(v)-|w|$. 
%We use the convention that $v_{-m^{w}(v)-|w|-0}^{-m^{w}(v)-|w|-1}=\emptyset$.
This notation is not ambiguous since once we fix   $v$ and $k$, we automatically fix  $v_{-m^{w}(v)}^{-1}$ and $v_{-m^{w}(v)-|w|-k}^{-m^{w}(v)-|w|-1}$.
Now, let us introduce the partition which is illustrated in the upper part of  Figure \ref{fig:generalpartition}. Define for any $a\in A$
\[
\alpha(a):=\inf_{\underline{z}}P(a|\underline{z}),
\]
and the collection of  intervals $\{I(a)\}_{a\in A}$, each one having length $|I(a)|=\alpha(a)$. For any $v\in A^{\star}\cup A^{-\mathbb{N}}$ such that $m^{w}(v)<+\infty$ we define the collection of intervals  
\[
I^{w}(a,v,k)\,\,\,,\,\,\,a\in A\,\,,\,k=0,\ldots,|v|-m^{w}(v)-|w|,
\]
each one having length 
\[
|I^{w}(a,v,k)|=\alpha^{w}(a,v,k)-\alpha^{w}(a,v,k-1)
\]
for $k\geq1$, and 
\[
|I^{w}(a,v,0)|=
\alpha^{w}(a,v,0)-\alpha(a).
\]
Suppose now we are given an entire past $\underline{z}\in A^{-\mathbb{N}}(w)$, and glue these intervals in the following order
\[
I(1),I(2),\ldots,I^{w}(1,\underline{z},0),I^{w}(2,\underline{z},0),\ldots,I^{w}(1,\underline{z},1),I^{w}(2,\underline{z},1)\ldots
\]
in  such a way that the left extreme of $I(1)$ coincides with $0$, and the left extreme of each intervals coincides with the right extreme of the preceding interval. What we obtain, by the convergence (\ref{ca}), is  a partition of $[0,1[$ such that for any $a\in A$
\[
Leb\left(I(a)\cup\bigcup_{k\geq0}I^{w}(a,\underline{z},k)\right)=P(a|\underline{z}),
\]
where $Leb$ denote the Lebesgue measure on $[0,1[$.
It is important to notice that for any  $a$, $k$ and $\underline{z}\in A^{-\mathbb{N}}(w)$, we can construct the interval $I^{w}(a,\underline{z},k)$ knowing only the suffix $z_{-(k+|w|+m^{w}(\underline{z}))}^{-1}$.
\begin{figure}[htp]
\centering
\includegraphics[scale=0.9]{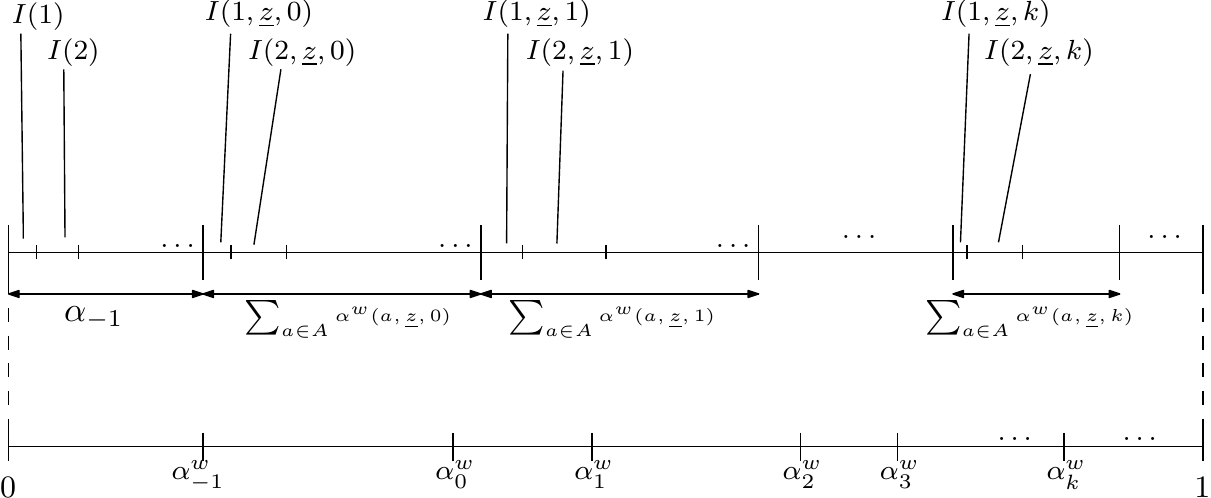}
\caption{Illustration of  the first partition (upper part) for a given past $\underline{z}\in A^{-\mathbb{N}}(w)$, and of the second partition (lower part) which does not depend on the past.}
\label{fig:generalpartition}
\end{figure}

\subsubsection{Definition of the second partition of [0,1[}
Observe that $\{\alpha^{w}_{k}\}_{k\geq0}$ is a $[0,1]$-valued non-decreasing sequence which converges to $1$ as $k$ diverges. It follows that denoting $\alpha^{w}_{-1}:=\sum_{a\in A}\alpha(a)$, and using the convention that $\alpha^{w}_{-2}=0$, the sequence of intervals $\{[\alpha^{w}_{k-1},\alpha^{w}_{k}[\}_{k\geq-1}$ constitutes a partition of $[0,1[$.
%\begin{equation}\label{partitionalpha}
%\bigcup_{k\geq-1}[\alpha^{w}_{k-1},\alpha^{w}_{k}[=[0,1[.
%\end{equation}
This partition is illustrated in the lower part of Figure \ref{fig:generalpartition}. 

\subsubsection{Definition of the triplet} Let us introduce an i.i.d. chain ${\bf U}=(U_{i})_{i\in\mathbb{Z}}$ of random variables uniformly distributed in $[0,1[$. We denote by $(\Omega,\mathcal{F},\mathbb{P})$ the corresponding probability space. It is the only probability space we will consider all along this paper. 
We now introduce one triplet $\left(\{\lambda^{w}_{k}\}_{k\geq-1},
  \,\{p^{w}_{-1}(a)\}_{a\in
    A},\,\{(\tau^{w}_{k},p^{w}_{k})\}_{k\geq0}\right)$ that will
give the decomposition stated in Theorem \ref{theo1}. Define
\begin{itemize}
\item For any $k\geq-1$
\begin{equation}\label{1}
\lambda^{w}_{k}:=\mathbb{P}(U_{0}\in[\alpha^{w}_{k-1},\alpha^{w}_{k}[).
\end{equation}
\item For any $a\in A$
\begin{equation}\label{2}
p^{w}_{-1}(a):=\mathbb{P}(U_{0}\in I(a)[|U_{0}\in [0,\alpha^{w}_{-1}[)=\alpha(a)/\alpha^{w}_{-1}.
\end{equation}
\item For any $k\geq0$, let
\begin{equation}\label{3}
\tau^{w}_{k}:=A^{-\mathbb{N}}(\bar{w})\,\cup\,\bigcup_{i\geq0}\bigcup_{b_{-i}^{-1}\in  \mathcal{I}^{i}(\bar{w})}\bigcup_{c_{-k}^{-1}\in A^{k}}c_{-k}^{-1}\,w\,b_{-i}^{-1}\,\,,\,\,\,k\geq0,
\end{equation}
and for $v\in\tau^{w}_{k}$  (that is, $|v|=m^{w}(v)+|w|+k$) we put
\begin{equation}\label{4}
p^{w}_{k}(a|v):= \left\{\begin{array}{ll}
\Prob\G(U_{0}\in
I(a)\cup\bigcup_{l=0}^{k}I^{w}(a,v,l)\Big|U_{0}\in[\alpha^{w}_{k-1},\alpha^{w}_{k}[\D)
& \mbox{ if } m^{w}(v)<+\infty \vspace{.2cm} \\
 \frac{P(a|v)-\lambda_{-1}^{w}p^{w}_{-1}(a)}{1-\lambda^{w}_{-1}} &
\mbox{otherwise.}
\end{array} \right.
\end{equation}
% \begin{equation}\label{4}
% p^{w}_{k}(a|v):= \Prob\G(U_{0}\in I(a)\cup\bigcup_{l=0}^{k}I^{w}(a,v,l)\Big|U_{0}\in[\alpha^{w}_{k-1},\alpha^{w}_{k}[\D)
% \end{equation}
% if $m^{w}(v)<+\infty$ and 
% \begin{equation}\label{semw}
% p^{w}_{k}(a|v):=\frac{P(a|v)-\lambda_{-1}^{w}p^{w}_{-1}(a)}{1-\lambda^{w}_{-1}}
% \end{equation}
%  otherwise.
\end{itemize}

Two examples of  sequences of context trees $\{\tau^{w}_{k}\}_{k\geq0}$ on $A=\{1,2\}$ are given in Figures \ref{fig:sequencehu} and \ref{fig:sequencedeux}. The first one with $w=2$, and  the second one  with $w=12$.
%\begin{figure}[htp]
%\centering
%\includegraphics{sequenceun}
%\caption{}
%\label{fig:sequenceun}
%\end{figure}
\begin{figure}[htp]
\centering
\includegraphics[scale=0.6]{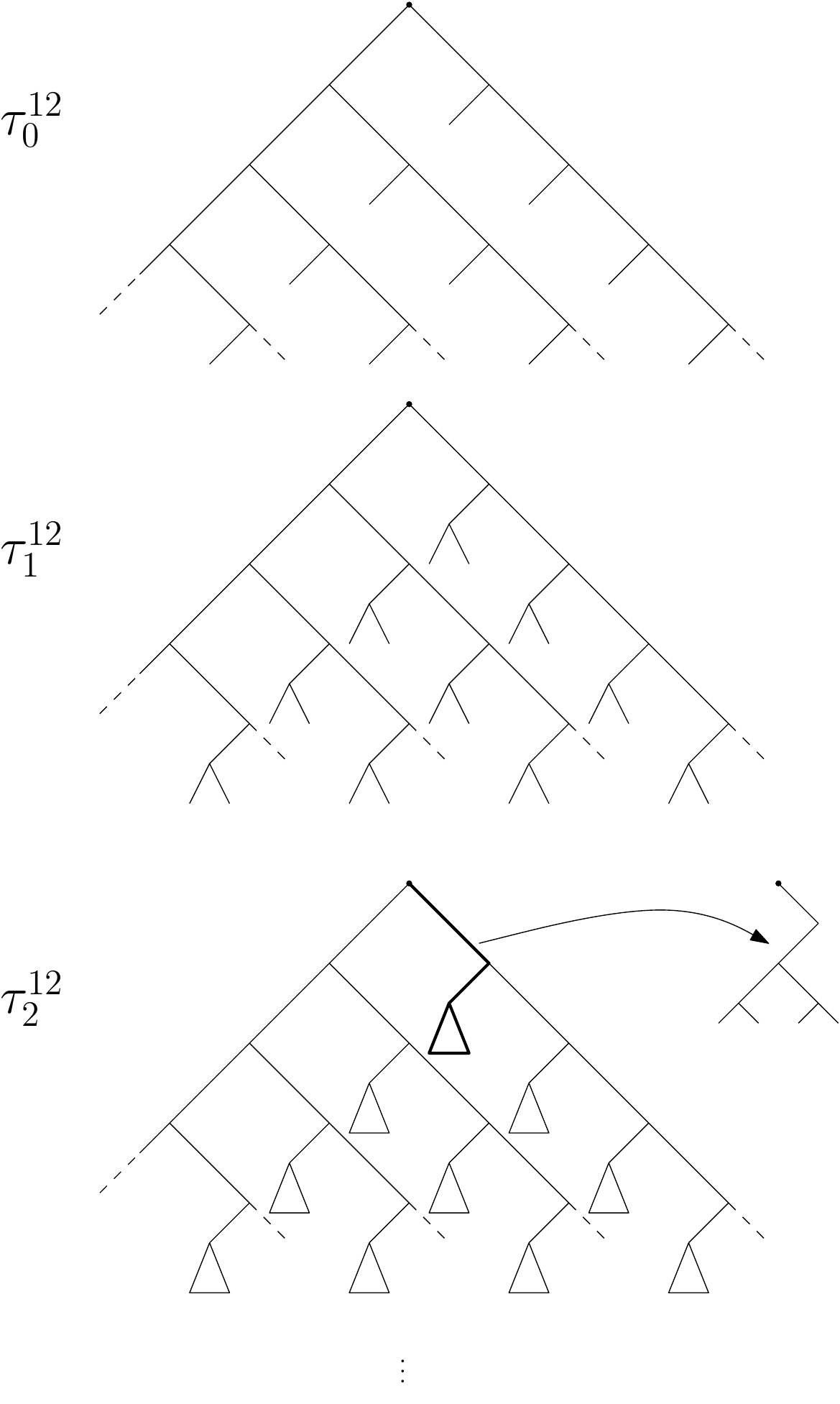}
\caption{}
\label{fig:sequencedeux}
\end{figure}

\subsection{Proofs of the results of this section}\label{sec:proof1}

\begin{proof}[Proof of Theorem \ref{theo1}]
What we have to prove is that equality (\ref{mixture}) holds with   the triplet $\left(\{\lambda^{w}_{k}\}_{k\geq-1}, \,\{p^{w}_{-1}(a)\}_{a\in A},\,\{(\tau^{w}_{k},p^{w}_{k})\}_{k\geq0}\right)$ introduced above.
On the one hand, using (\ref{ca}),   for any $a \in A$ and  $\underline{z}\in A^{-\mathbb{N}}(w)$ we have
\begin{equation}\label{lebesgue}
P(a|\underline{z})=\Prob(U_{0}\in I(a))+\mathbb{P}\left(U_{0}\in\bigcup_{k\geq0}I^{w}(a,\underline{z},k)\right),
\end{equation}
where the second term can be rewritten
\[
\sum_{k\geq0}\Prob(U\in [\alpha^{w}_{k-1},\alpha^{w}_{k}[)\Prob\G(U\in \bigcup_{l\geq0}I^{w}(a,\underline{z},l)\Big|U\in [\alpha^{w}_{k-1},\alpha^{w}_{k}[\D).
\]
%with the convention that $I(a)=I^{w}(a,\underline{z},-1)$ for any $a\in A$ and $\underline{z}$.
%On the other hand, by Lemma \ref{lemmaparticao2}, the sequence $(\alpha_{k})_{k\geq-1}$ converges to $1$ as $k$ diverges. Therefore, we can define the disjoint partition
%\[
%\bigcup_{k\geq-1}[\alpha_{k-1},\alpha_{k}[=[0,1[
%\]
%with the convention that $\alpha_{-2}=0$. 
On the other hand, by the definition of  $\alpha^{w}_{k}$, $k\geq-1$, we have $I(a)\subset[0,\alpha^{w}_{-1}[$ and  for any $k\geq0$
\[
[0,\alpha^{w}_{k}[\subset\bigcup_{a\in A}\bigcup_{l=0}^{k}I^{w}(a,\underline{z},l),
\]
it follows that for any $\underline{z}\in A^{-\mathbb{N}}(w)$, 
\begin{eqnarray*}
P(a|\underline{z}) & = & \lambda^{w}_{-1}\mathbb{P}(U_{0}\in
I(a)|U_{0}\in[0,\alpha^{w}_{-1}[) \\
 && +\sum_{k\geq0}\lambda^{w}_{k}\Prob\G(U\in I(a)\cup\bigcup_{l=0}^{k}I^{w}(a,\underline{z},l)\Big|U\in[\alpha^{w}_{k},\alpha^{w}_{k-1}[\D).
\end{eqnarray*}
It follows from (\ref{4}) that for any $\underline{z} \in
A^{-\mathbb{N}}$ and $a\in A$
\[
P(a|\underline{z})=\lambda^{w}_{-1}p^{w}_{-1}(a)+\sum_{k\geq0}\lambda^{w}_{k}p^{w}_{k}(a|c_{\tau^{w}_{k}}(\underline{z})).
\]
\end{proof}

\begin{proof}[Proof of Proposition \ref{prop}]
The fact that (\ref{one}) implies (\ref{two}) is clear since $A^{k}(w)\subset A^{k}$. Now let us show that (\ref{two}) implies (\ref{condition1}).
First, define for any $k\geq0$ and $i\geq0$
%\[
%\inf_{a_{-k}^{-1}\in A_{w}^{k}}\sum_{a\in A}\alpha^{w}\G(a,a_{-m^{w}(a_{-k}^{-1})}^{-1},a_{-k}^{-m^{w}(a_{-k}^{-1})-|w|-1}\D)\stackrel{k\rightarrow+\infty}{\longrightarrow}1
%\]
%and we observe that
\[
\alpha^{w}_{k,i}:=\inf_{b_{-i}^{-1}\in \mathcal{I}^{i}(\bar{w})}\inf_{c_{-k}^{-1}\in A^{k}}\sum_{a\in A}\inf_{\underline{z}}P\left(a\big|b_{-i}^{-1}\,w\,c_{-k}^{-1}\,\underline{z}\right)
\]
and observe that
\[
\inf_{a_{-(k+i+|w|)}^{-1}\in A^{k+i+|w|}(w)}\sum_{a\in A}\inf_{\underline{z}}P\G(a\Big|a_{-(k+i+|w|)}^{-1}\,\underline{z}\D)\leq\alpha^{w}_{k,i}.
\]
Thus, condition (\ref{condition1}) implies at the same time that
$\alpha^{w}_{k,i}\stackrel{k\rightarrow+\infty}{\longrightarrow}1$ for
any fixed $i$ and
$\alpha^{w}_{k,i}\stackrel{i\rightarrow+\infty}{\longrightarrow}1$ for
any fixed $k$.  Since $\alpha^{w}_{k,i}$ belongs to $[0,1]$ for any
$k$ and $i$, it follows that
$\inf_{i\geq0}\alpha^{w}_{k,i}$($:=\alpha_{k}^{w}$) also goes to $1$
as $k$ diverges.

\end{proof}

\section{Second main result: perfect simulation}\label{algorithm}

In this section, we present the perfect simulation algorithm and state the second main result of this paper, which gives sufficient conditions for the algorithm to stop after a  $\mathbb{P}$-a.s. finite number of steps. 
\subsection{Explaining how our algorithm works}

The algorithm works as a mixture of the algorithm presented in
\cite{gallo/2009} and the one of
\cite{comets/fernandez/ferrari/2002}. 

 Assume that the set
\[
\mathcal{E}:=\{a\in A:\inf_{\underline{z}}P(a|\underline{z})>0\}
\]
is not empty, and let $\mathcal{E}^{\star}$ denotes the set of finite
strings of symbols of $\mathcal{E}$.

Consider a transition probability kernel $P$ satisfying the condition
of Theorem \ref{theo1} with reference string $w \in
\mathcal{E}^{\star}$. In order to simplify the notation, we will omit
the superscript $w$ in most of the quantities that depend on this
string.
%This is the distance to the last occurrence of $w$ in $a_{-k}^{-1}$. 

%\begin{figure}[htp]
%\centering
%\includegraphics{generalpartition}
%\caption{Illustration of the partition of $[0,1[$ with the disjoint intervals $\{J(a|\emptyset)\}_{a\in \mathcal{E}}$ and $\{J(a|v)\}_{a\in \mathcal{A}}$ for some $v\in\tau$.}
%\label{fig:partitionof[0,1]}
%\end{figure}

We want to get a deterministic measurable function
$X:[0,1[^{\mathbb{Z}}\rightarrow A^{\mathbb{Z}}$, ${\bf U}\mapsto
X({\bf U})$ such that the law $\mathbb{P}( X({\bf U})\in \cdot )$ is
compatible with $P$ in the sense of (\ref{compa}). The idea is to use
the sequence ${\bf U}$ together with the partitions of $[0,1[$
introduced before (and illustrated in Figure
\ref{fig:generalpartition}) to mimic the two steps procedure we
described in Section \ref{altern}.  

In particular, for any $n\in\mathbb{Z}$, we put
$\XU_{n}=a$ whenever $U_{n}\in I(a)$.
Suppose that for some time index $n\in\mathbb{Z}$ there exists a
string $a_{-k}^{-1}\in A^{k}$ such that $U_{n-i}\in
I(a_{-i}),\,i=1,\ldots,k$, in this case, we put
\[
\XU^{n-1}_{n-k}=a_{-k}^{-1}.
\] 
We say that this sample  has been \emph{spontaneously} constructed. Now suppose $U_{n}\in[\alpha_{l-1},\alpha_{l}[$ for some $l\geq0$. This means that we pick up the context tree $\tau_{l}$ in the countable mixture representation of $P$, and look whether or not there exists a context in $\tau_{l}$ which is suffix of $\XU^{n-1}_{n-k}=a_{-k}^{-1}$. If there exists such a context, then we put
\[
\XU_{n}=\sum_{a\in A}a.{\bf 1}\left\{U_{n}\in \bigcup_{j=0}^{l}I(a,a_{-k}^{-1},j)\right\}.
\]
If there is no such context (we will write $c_{\tau_{l}}(a^{-1}_{-k})=\emptyset$) we cannot construct the state $\XU_{n}$: we need more knowledge of the past. 
In the first case, $\XU^{n}_{n-k}$ has been constructed independently of $U_{-\infty}^{n-k-1}$ and $U_{n+1}^{+\infty}$. 
%Fix $l\geq0$ and assume that for each $j$, $j=0,\ldots,l$, the symbols assigned to $\XU_{n-k}^{n+j-1}$ and the value of $U_{n+j}$ have been such that we were in situations (i) or (ii). This means that we can construct $\XU_{n-k}^{n+l}$ without specifying the past symbols $\XU_{-\infty}^{n-k-1}$.
Now suppose we want to construct $\XU_{0}$. We generate backward in
time the $U_{i}$'s until the first time $k\leq0$ such that we can
perform the above construction from time $k$ up to time $0$ using only
$U_{k}^{0}$. A priori, there is no reason for $k$ to be
finite. Theorem \ref{theo2} gives sufficient conditions for $k$ to be
finite $\mathbb{P}$-almost surely.

To formalize what we just said, let us define for any $u\in[0,1[$
\[
\ell(u)=\sum_{k\geq-1}k.{\bf 1}\{ u\in[\alpha_{k-1},\alpha_{k}[\}.
\]
%To formalize the preceding discussion, we introduce the measurable function $L:[0,1[\times (\emptyset\cup A^{\star}\cup A^{-\mathbb{N}})\rightarrow \mathbb{N}$ which is defined for any $a_{m}^{n}\in A^{\star}\cup A^{-\mathbb{N}}$, $-\infty\leq m\leq n+1$ by
%\begin{equation}
%L(u,a_{m}^{n})=\left\{
%\begin{array}{ccc}
%c_{\tau_{\ell(u)}}&\textrm{if}&\ell(u)\geq0\\
%0&\textrm{if}&\ell(u)=-1.
%\end{array}
%\right.
%\end{equation}
%%\begin{equation}
%%L(u,a_{m}^{n})=\left\{
%%\begin{array}{ccc}
%%\ell(u)+|w|+m^{w}(a_{m}^{n})&\textrm{if}&\ell(u)\geq0\\
%%0&\textrm{if}&\ell(u)=-1.
%%\end{array}
%%\right.
%%\end{equation}
%with the conventions that $a_{n+1}^{n}=\emptyset$ and $m^{w}(\emptyset)=+\infty$.
%At each time $i\in\mathbb{Z}$, if $L(U_{i},\XU_{i-k}^{i-1})>0$, it represents the size of the suffix of $\XU_{i-k}^{i-1}$ we need to know in order to construct $\XU_{i}$. 
%% If we only constructed $\XU_{i-k}^{i-1}$ and $L(U_{i},\XU_{i-k}^{i-1})>k$, then we cannot construct $\XU_{i}$: we need to know more past values of $X({\bf U})$. 
%$L(U_{i},\XU_{i-k}^{i-1})=0$ means that $\XU_{i}$ is spontaneously constructed. 
By Theorem \ref{theo1}, $\ell(U_{i})=-1$ means that we can choose the
state of $X({\bf U})_{i}$ according to distribution $p_{-1}(\cdot)$,
and independently of everything else. On the other hand,
$\ell(U_{i})=l\geq0$  means that we have to use the context tree
$(\tau_{l},p_{l})$ in order to construct the state of $X({\bf
  U})_{i}$. In particular, we  recall that for any $l\geq0$ the size
of the context $c_{\tau_{l}}(a_{m}^{n})$ is $m^{w}(a_{m}^{n})+|w|+l$. \\

\noindent One of the inputs for Algorithm 1 is the \emph{update
  function} $F$. It is a measurable function
$F:[0,1[\times(\emptyset\cup A^{\star}\cup A^{-\mathbb{N}})\rightarrow
A\cup\{\star\}$ which uses the part of the past we already know
and the uniform random variable to compute the present state. It is
defined as follows: for any $a_{m}^{n}\in \emptyset\cup A^{\star}\cup
A^{-\mathbb{N}}$, with $-\infty<n<+\infty$ and $-\infty\leq m\leq
n+1$,
\begin{equation}
F(u,a_{m}^{n}):=\left\{
\begin{array}{ll}
\sum_{a\in A}a.{\bf 1}\{u\in I(a)\}&\textrm{if }\ell(u)=-1\\
\sum_{a\in A}a.{\bf 1}\left\{u\in
  \bigcup_{k=0}^{\ell(u)}I(a,a_{m}^{n},k)\right\}&\textrm{if }\ell(u)\geq0\,\,\textrm{and}\,\,c_{\tau_{\ell(u)}}(a_{m}^{n})\neq\emptyset\\
\star&\textrm{otherwise}
\end{array}
\right.
\end{equation}
%
%\begin{equation}\label{F}
%F(u,a_{m}^{n})=\sum_{a\in A}a.{\bf 1}\left\{I^{w}(a,a_{m}^{n},k)\right\}+\star\mathbbm{1}\{u\in [\#\mathcal{E}\epsilon,1[,\, c_{\tau}(a_{m}^{n})=\emptyset\},
%\end{equation}
with the conventions that $a_{n+1}^{n}=\emptyset$  and  for any context tree $\tau$, $c_{\tau}(\emptyset)=\emptyset$.
%This is our \emph{update function}: it is the function that compute the present state using the known past and the uniform random variable. 
When we consider an infinite past $\underline{z}\in A^{-\mathbb{N}}(w)$,  we have by (\ref{lebesgue}), for any $u\in[0,1[$
\begin{equation}\label{lebe}
\mathbb{P}(F(u,\underline{z})=a)=\mathbb{P}\left(u\in I(a)\cup \bigcup_{k\geq0}I^{w}(a,\underline{z},k)\right)=P(a|\underline{z}).
\end{equation}
When the  update function returns the symbol $\star$, it means that we
do not have sufficient knowledge of the past to compute the present state. 

We define, for any $m\leq n$, the $\mathcal{F}(U_{m}^{n})$-measurable function $\mathcal{L}: [0,1[^{n-m+1}\rightarrow \{0,1\}$ which takes value $1$ if, and only if, we can construct $\XU_{m}^{n}$ independently of $U_{-\infty}^{m-1}$ and $U_{n+1}^{+\infty}$ using the construction described above. Formally
\[
\{\mathcal{L}(U_{m}^{n})=1\}:=\bigcup_{a_{m}^{n}\in A^{n-m+1}}\bigcap_{i=m}^{n}\{F(U_{i},a_{m}^{i-1})=a_{i}\}.
\]
Finally, for any $-\infty<m\leq n\leq+\infty$, we define
 the \emph{regeneration time} for the window $[m,n]$ as  the first
 time before $m$ such that the construction described above is
 successful until time $n$, that is
\begin{equation}\label{formaldef}
\theta[m,n]:=\max\{k\leq m:\mathcal{L}(U_{k}^{n})=1\}
\end{equation}
with the convention that  $\theta[m]:=\theta[m,m]$.

\subsection{The algorithm}

This algorithm takes as ``input'' two integers $-\infty<m\leq
n<+\infty$ and the update function $F$, and returns as ``output'' the
regeneration time $\theta[m,n]$ and the constructed sample $[X({\bf
  U})]_{\theta[m,n]}^{n}$. The function $F$ contains all the
information we need about $P$, and we suppose that it is already
implemented in the software used for programing the algorithm.

\begin{algorithm}[h]
\caption{Perfect simulation algorithm of the sample $\XU_{m}^{n}$} 
\begin{algorithmic}[1]
%\vspace{0.1cm}
%\STATE {\it Input:} $m$, $n$; {\it Output:} $\theta$, $(X_{\theta},\ldots,X_{0})$
\STATE {\it Input:} $m$, $n$, $F$; {\it Output:} $\theta[m,n]$, $(\XU_{\theta[m,n]},\ldots,\XU_{n})$
%\vspace{0.1cm}
%\STATE  $m \leftarrow 1$, $n \leftarrow 0$, $\theta \leftarrow 0$, $X_{0}=\star$ 
%\STATE  $\XU_{m}^{n}=\star^{n-m+1}$ 
%\vspace{0.1cm}
%\WHILE{$X_{0}=\star$} 
\STATE  Sample $U_{m},\ldots,U_{n}$ uniformly in $[0,1[$\\
%\vspace{0.1cm}
\STATE $i \leftarrow m$, $B=\{m,\ldots,n\}$, $\theta[m,n] \leftarrow m$, $\XU_{m}^{n}\leftarrow\star^{n-m+1}$\\
%\vspace{0.1cm}
\WHILE{$F(U_{i},\XU_{m}^{i-1})\in A$ and $B\neq \emptyset$}                           
%\vspace{0.1cm}
\STATE $\XU_{i}\leftarrow F(U_{i},\XU_{m}^{i-1})$\\
%\vspace{0.1cm}
\STATE $B\leftarrow B\setminus \{i\}$\\
\STATE $i\leftarrow i+1$
%\vspace{0.1cm}
\ENDWHILE\\
%\vspace{0.1cm}
\STATE $i\leftarrow m$
\WHILE{$B\neq\emptyset$} 
%\vspace{0.1cm}
\STATE $i\leftarrow i-1$\\
\STATE $B\leftarrow B\cup\{i\}$\\
%\vspace{0.1cm}
\STATE Sample $U_{i}$ uniformly in $[0,1[$\\
%\vspace{0.1cm}
%define $Y_{m}=F(U_{m},\emptyset)$
%\WHILE{$Y_{m}=\star$}
\WHILE {$U_{i}\in[\#\mathcal{E}\epsilon,1[$}
%\vspace{0.1cm}
\STATE $i\leftarrow i-1$\\
\STATE $B\leftarrow B\cup\{i\}$\\
%\vspace{0.1cm}

%\STATE $j\leftarrow i$\\

%\vspace{0.1cm}
\STATE Sample $U_{i}$ uniformly in $[0,1[$\\
%\vspace{0.1cm}
%define $Y_{m}=F(U_{m},\emptyset)$
\ENDWHILE\\
%\vspace{0.1cm}
%$X_{m}\leftarrow Y_{m}$
\STATE $\XU_{i}\leftarrow F(U_{i},\emptyset)$
\STATE $B\leftarrow B\setminus\{i\}$\\
\STATE $t\leftarrow \min B$
%\vspace{0.1cm}
%\WHILE{$L(X_{m}^{n})=1$ and $n<0$}
%\WHILE{$L(X_{m}^{n},U_{n+1})\neq\Lambda$ and $n<0$}
%\STATE $X_{n+1}=f\G(U_{n+1},X_{n-L(X_{m}^{n},U_{n+1})}^{n}\D)$\\
\WHILE{$F(U_{t},\XU_{i}^{t-1})\in A$ and $B\neq\emptyset$}
%\vspace{0.1cm}
\STATE $\XU_{t}\leftarrow F(U_{t},\XU_{i}^{t-1})$\\
\STATE $B\leftarrow B\setminus\{t\}$\\
%\vspace{0.1cm}
\STATE $t\leftarrow \min B$
%\vspace{0.1cm}
\ENDWHILE\\
%\vspace{0.1cm}
\ENDWHILE\\
%\vspace{0.1cm}
\STATE $\theta[m,n]\leftarrow i$
%\vspace{0.1cm}
\RETURN $\theta[m,n]$, $(\XU_{\theta[m,n]},\ldots, \XU_{n})$
%\vspace{0.1cm}
%\RETURN $\theta$, $(X_{\theta},\ldots, X_{0})$
\end{algorithmic}
\end{algorithm}

At each time, the set $B$ contains the sites that remains to be
constructed. At first $B = \{m, \ldots, n\}$ and a forward procedure
(lines 2--8) tries to construct $[X({\bf U})]_{m}^n$ using $U_m,
\ldots, U_n$. If it succeeds, then the algorithm stops and returns
$\theta[m,n]=m$ and the constructed sample. If it fails, $B$ is not
empty and a backward procedure (``while loop'': lines 10--27) begins.
In this loop, each time the algorithm cannot construct the next site
of $B$, it generates a new uniform random variable backward in
time. At each new generated random variable, the algorithm attempts to
go as far as possible in the construction of the remaining sites of
$B$ using the uniform that have been previously generated. Theorem
\ref{theo2} gives sufficient conditions for this procedure to stop
after a finite number of steps.

\subsection{Statement of the second main theorem}

\begin{theo}\label{theo2}
Consider a kernel $P$ satisfying the conditions of Theorem \ref{theo1} for some string $w\in \mathcal{E}^{\star}$. If the sequence $(\alpha^{w}_{k})_{k\geq0}$  defined by (\ref{condition1}) satisfies
\[
\sum_{k\geq0}(1-\alpha^{w}_{k})<+\infty \,\,\,\,\left(\textrm{or, equivalently $\prod_{k\geq0}\alpha_{k}^{w}>0$}\right)
\]
then Algorithm 1 stops after a $\mathbb{P}$-a.s. finite number of steps for any $-\infty<m\leq n\leq+\infty$. 
%\[
%\Prob(\theta[m,n]>-\infty)=1.
%\]
Moreover, for any $n\in\mathbb{Z}$
\begin{equation}
\label{eq:tail}
\sum_{l\geq0}\Prob(\theta[0]<-l)<+\infty.
\end{equation}
\end{theo}

\begin{coro} \label{coro:2}
  The output of Algorithm 1 is a
  sample of the unique stationary chain compatible with $P$. Moreover,
   there exists a sequence of random times
  ${\bf T}={\bf T}({\bf U})$ which \emph{splits} the realization ${\bf
    X}$ into i.i.d. pieces. More specifically, the random strings
  $(\XU_{T_{i}},\ldots,\XU_{T_{i+1}-1})_{i\neq0}$ are i.i.d. and have
  finite expected size. \end{coro}

The proof of Corollary \ref{coro:2} using the CFTP algorithm and
Theorem \ref{theo2} is essentially the same as
\cite{comets/fernandez/ferrari/2002} (Proposition 6.1, Corollary 4.1
and Corollary 4.3). We omit these proofs in the present work and just
mention the main ideas. The existence statement follows once we
observe that Theorem \ref{theo2} implies that one can construct a
bi-infinite sequence $\bf{X}$ verifying for any $n \in \mathbb{Z}$,
$X_n = F(U_n, X_{-\infty}^{n-1})$. By (\ref{lebe}), this chain is
therefore compatible in the sense of (\ref{compa}). It is stationary by
construction. The uniqueness statement follows from the loss of memory
the chain inherits because of the existence of almost surely finite
regeneration times. The regeneration scheme follows from
(\ref{eq:tail}). 

\section{Proof of Theorem \ref{theo2}}\label{laprova}

%and denote 
%\[
%\ell_{i}:=\ell(U_{i}).
%\]
%Assume for one instant that $X_{i-k}^{i-1}$ has been constructed using our algorithm. We recall that when $\ell_{i}\geq0$, it tells us how much further the last occurrence of $w$ in $X_{i-k}^{i-1}$ we have to look at in order to decide $X_i$.

Let us explain what are the main steps of this proof.  To study
directly the random variable $\theta[m,n]$ is complicated, because it
depends on the construction of the states of $X({\bf U})$: in order to
construct the next state of the chain, we may need to know the
distance to the last occurrence of $w$ in the constructed sample. The
idea is to introduce, first, a new random variable
$\bar{\theta}[m,n]$, defined by (\ref{bartheta}), which can be used to
define a lower bound for $\theta[m,n]$. The advantage of
$\bar{\theta}[m,n]$ is that its definition depends on the
reconstructed sample only through the spontaneous occurrences of
$w$. Section \ref{defining1} is dedicated to the definition of this
new random variable. After that, the main problem is transformed into
the problem of showing that $\bar{\theta}[m,n]$ is itself
$\mathbb{P}$-a.s. finite. To solve this new problem, we study in
Section \ref{defining2} an auxiliary process ${\bf D}^{(0)}$, defined
by (\ref{key}). The probability of return to $0$ of ${\bf D}^{(0)}$ is
related to the distribution of $\bar{\theta}[0,n]$ through equation
(\ref{yo}). The conclusion of the proof is done in Section
\ref{defining3}, by studying the chain ${\bf D}^{(0)}$ (Lemma
\ref{aah}).

\subsection{Definition of a new random variable $\bar{\theta}[m,n]$}\label{defining1}

Define the i.i.d. stochastic chain ${\bf Z}$ which takes value
$Z_{i}=a$ if $U_{i}$ belongs to $I(a)$, and $Z_{i}=\star$
otherwise. This chain takes in account only the symbols which appear
spontaneously in $X({\bf U})$: $[X({\bf U})]_{i}=a$ whenever
$Z_{i}=a$, and in particular $[X({\bf U})]_{i-|w|+1}^{i}=w$ whenever
$Z_{i-|w|+1}^{i}=w$, for any $i\in\mathbb{Z}$.  We also define the
distance to the last spontaneous occurrence of $w$ in ${\bf Z}$ before
time $i$ as
\[
m_{i}=\inf\G\{k\geq0:Z_{i-k-|w|}^{i-k-1}=w\D\}.
\]
Suppose we already constructed a sample $[X({\bf
  U})]_{-k}^{-1}$. Since for any $n$, $[X({\bf U})]_{n-|w|+1}^{n}=w$
whenever $Z_{n-|w|+1}^{n}=w$, it follows that $m_{0}$ is   larger or
equal than $m^{w}([X({\bf U})]_{-k}^{-1})$. Denote
$\ell_{i}:=\ell(U_{i})$ and define the random variable 
\begin{equation}\label{bark}
L_{i}=\left\{
\begin{array}{ccc}
0&\textrm{ if }&Z_{i}\in \mathcal{E},\\
m_{i}+|w|+\ell_{i} &\textrm{ otherwise}.
\end{array}\right.
\end{equation}
Then, whenever $L_{0}>0$, it is larger or equal than $m^{w}([X({\bf
  U})]_{-k}^{-1}) + |w| + l_0$ which is the size of the suffix of
$[X({\bf U})]_{-k}^{-1} $ we need to know in order to construct
$[X({\bf U})]_{0}$. Before we define $\bar{\theta}[m,n]$, let us
introduce an intermediary random variable $\theta'[m,n]$ which depends
on the spontaneous occurrences of $w$. For
any $-\infty\leq m\leq n\leq+\infty$
\begin{equation}\label{thetalina}
\theta'[m,n]=\max\{k\leq m:L_{i}\leq i-k\,,\,\,i=k,\ldots,n\}.
\end{equation}
Associate to each site $i\in\{\theta'[m,n],\ldots,n\}$ an arrow going from time $i$ to time $i-L_{i}$. Definition (\ref{thetalina}) says that no arrow will pass time $\theta'[m,n]$, meaning that we can construct $[X({\bf U})]_{\theta'[m,n]}^{n}$ knowing only $U_{\theta'[m,n]}^{n}$. Therefore, $\mathcal{L}(U_{\theta'[m,n]}^{n})=1$. Since $\theta[m,n]$ is the maximum over all time indexes $k\leq m$ such that $\mathcal{L}(U_{k}^{n})=1$, it follows that
\[
\theta'[m,n]\leq\theta[m,n].
\]
%
%The rest of this section is quite technical, and we give it without further details, since it is in every points similar to the proof of Theorem 5.1 in \cite{gallo/2009}, the only difference being in the definition of $\ell$. 

The definition of $\bar{\theta}[m,n]$ is done using   the following
rescaled quantities. Consider the chain $\bar{\bf Z}$ defined by
\begin{equation}\label{barz}
\bar{Z}_{m}=\left\{
\begin{array}{ccc}
1&\textrm{ if }&U_{m|w|-i+1}\in I(w_{-i}),\,\,\,i=0,\ldots,|w|-1\\
\star &\textrm{ otherwise}.
\end{array}\right.
\end{equation} 
and the rescaled function
\[
\bar{\ell}_{i}:=\left\lceil\frac{\sup\{\ell_{j}:j=(i-1)|w|+1,\ldots,i|w|\}}{|w|}\right\rceil
\]
where for any $r\in \mathbb{R}$, $\lceil r\rceil$ denotes the smaller integer which is larger or equal to $r$. Using these rescaled quantities, we define the corresponding random variables
\[
\bar{m}_{i}=\inf\G\{k\geq0:\bar{Z}_{i-k-1}=1\D\}
\]
which is the distance to the last occurrence of $1$ in $\bar{Z}_{-\infty}^{i-1}$ and 
%satisfies the inequality $m_{i|w|}\leq (\bar{m}_{i}+1)|w|-1$, and 
\begin{equation}\label{bark}
\bar{L}_{i}=\left\{
\begin{array}{ccc}
0&\textrm{ if }&\bar{Z}_{i}=1,\\
\bar{m}_{i}+1+\bar{\ell}_{i} &\textrm{ otherwise}.
\end{array}\right.
\end{equation}
%Notice that for any $i\in\mathbb{Z}$ and $j=(i-1)|w|+1,\ldots,i|w|$, we have 
%\[
%m_{j}\leq(\bar{m}_{i}+1)|w|-1-(i|w|-j)
%\]
%and $\ell_{j}\leq |w|\bar{\ell}_{i}$. Now observe that these inequalities imply that for any $i\in\mathbb{Z}$ and $j=(i-1)|w|+1,\ldots,i|w|$
%\[
%L_{j}\leq|w|-1-(i|w|-j)+|w|\bar{L}_{i}.
%\]
The utility of all these new definitions lays in the fact (which is proven in details in \cite{gallo/2009}, the only difference being the definition of the function $\bar{\ell}_{i}$) that
%whenever $\bar{L}_{i}>0$,
%\[
%\bar{L}_{i}\geq L_{j}
%\]
%for $j=i-|w|+1,\ldots, i-1$. On the other hand, if $\bar{L}_{i}=0$, then $L_{j}=0$ for $j=i-|w|+1,\ldots, i-1$. Finally, defining 
the rescaled random variable
\begin{equation}\label{bartheta}
\bar{\theta}[0,n]:=\max\{k\leq 0:\bar{L}_{i}\leq i-k\,,\,\,i=k,\ldots, n\}
\end{equation}
satisfies the inequality
\[
(\bar{\theta}[0,n]-1)|w|+1\leq \theta'[0,n|w|]\leq\theta[0,n|w|]
\]
for any $n\geq0$.  All we need to study now is the
distribution of $\bar{\theta}[0,n]$. This is done in Sections
\ref{defining2} and \ref{defining3}.

To clarify the relationship between $\bar{\theta}[0,n]$,
$\theta'[0,n|w|]$ and $\theta[0,n|w|]$ let us give a concrete
example.

\subsection*{Example}

Consider a kernel $P$ satisfying the conditions of Theorem
\ref{theo2}, with a string $w$ having length $|w|=3$.  Assume we are
given a sample $U_{-38}^{6}$ (which we do not specify) to which
correspond two samples $Z_{-38}^{6}$ and $\bar{Z}_{-12}^{2}$, with two
sequences of arrows $L_{-38}^{6}$ and $\bar{L}_{-12}^{2}$. The sample
$Z_{-38}^{6}$ together with the sequence of arrows $L_{-38}^{6}$ are
illustrated in the lower part of Figure \ref{fig:combine}, and the
sample $\bar{Z}_{-12}^{2}$ together with the sequence of arrows
$\bar{L}_{-12}^{2}$ are illustrated in the middle part of Figure
\ref{fig:combine}. The loops mean that $L_{i}=0$ or $\bar{L}_{i}=0$.

We have sufficient information to determine lower bounds for
$\theta[0,6]$. In fact, we can see on the lower sequence that no arrow
merging from $i\in\{-29,\ldots,6\}$ go further time $-29$, and that
$-29$ is the first time in the past satisfying this. Therefore
$\theta'[0,6]=-29$. This is a first lower bound for
$\theta[0,6]$. Another lower bound can be obtained looking at the
sequence in the middle of the figure, no arrow goes further time $-12$,
meaning that $\bar{\theta}[0,2]=-12$. Then, as we said,
$\bar{\theta}[0,2]$ satisfies inequality (\ref{bartheta}), this allows
us to use the lower bound $(\bar{\theta}[0,2]-1)|w|+1=-38$ for
$\theta[0,6]$.

\subsection{A new auxiliary chain for the study of $\bar{\theta}[0,n]$}\label{defining2}

%It is defined by $$Similar inequalities are explained in details in \cite{gallo/2009}, and for this reason we state them without further details. 
%In the present case, we only introduced them to define the main object of importance in this section:  the auxiliary chain ${\bf D}^{n}$, $n\in\mathbb{Z}$.  
For any $n\in\mathbb{Z}$, the chain ${\bf D}^{(n)}$   takes values $D_{i}^{(n)}=0$ for any $i\leq n$ and
\begin{equation}\label{key}
D_{i}^{(n)}=(i-i^{(n)}-\bar{L}_{i})\vee 0\,\,,\,\,\,\forall i\geq n+1,
\end{equation}
where $i^{(n)}:=\max\{l< i:D_{l}^{(n)}=0\}$. The behavior of this chain is explained in the upper part of Figure \ref{fig:combine}. But it is clear from its definition that if $D^{(i)}_{n}>0$ for $n=i+1,\ldots,k$ for some $k\geq i+1$, then, in the process $\bar{\bf L}$, no arrow merging from $\{i+1,\ldots ,k\}$ passes time $i+1$, meaning that $\bar{\theta}[i+1,k]=i+1$. More generally, the sequence of chains $\{{\bf D}^{(n)}\}_{n\in\mathbb{Z}}$ satisfies the equation
\[
\G\{\bar{\theta}[0,n]<-l\D\}=\bigcap_{i=-l-1}^{-1}\bigcup_{k=i+1}^{n}\G\{D_{k}^{(i)}=0\D\}.
\]
\begin{figure}
\centering
\includegraphics[scale=0.9]{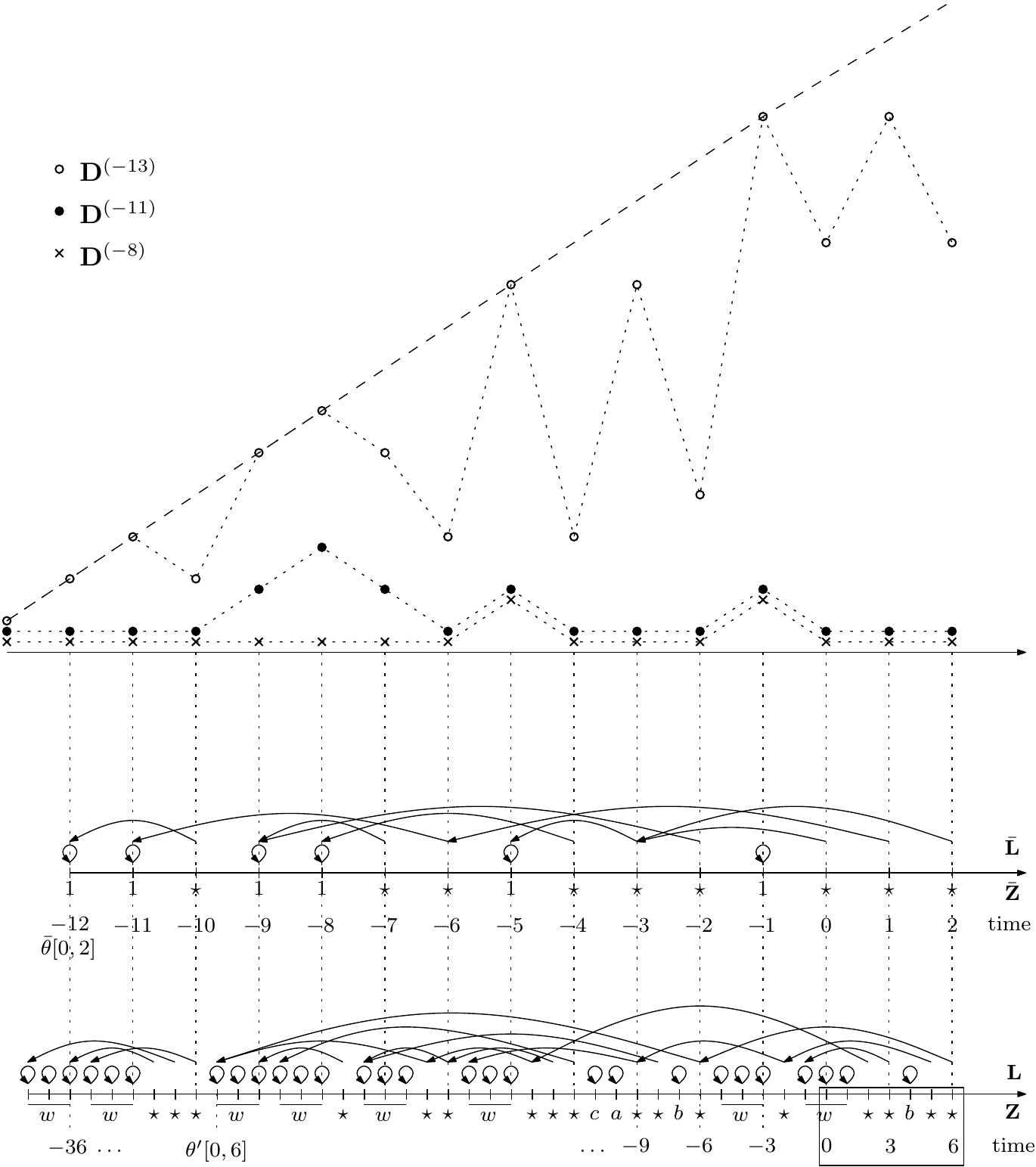}
\caption{Illustration of the inequalities of (\ref{bartheta})  (samples of ${\bf Z}$, ${\bf L}$, $\bar{\bf Z}$ and $\bar{\bf L}$) and of the behavior of the chain ${\bf D}^{(i)}$, for $i=-13,-11,-8$ constructed using the samples of $\bar{\bf Z}$ and $\bar{\bf L}$.}
\label{fig:combine}
\end{figure}
In fact, we can show in a similar way as in Section 6 of \cite{gallo/2009}, the only difference being the definition of the function $\bar{\ell}_{i}$, that
\begin{equation}\label{yo}
\Prob(\theta[0,n]<-l)\leq\sum_{k=\G\lfloor\frac{l}{|w|}\D\rfloor }^{\G\lfloor\frac{l}{|w|}\D\rfloor+\G\lceil \frac{n}{|w|}\D\rceil}\Prob(D_{k}^{(0)}=0)
\end{equation} 
where $\lfloor r\rfloor$ denotes the integer part of $r$. This
inequality relates the distribution we are interested in with the
probability of return to $0$ of the chain ${\bf D}^{(0)}$.

\subsection{Finishing the proof of Theorem \ref{theo2}}\label{defining3}
Owning to inequality (\ref{yo}), the proof of Theorem \ref{theo2} is done if we prove the following lemma. 

\begin{lemma}\label{aah}
Under the conditions of Theorem \ref{theo2}
%\[
%\sum_{k=\G\lfloor\frac{l}{|w|}\D\rfloor }^{\G\lfloor\frac{l}{|w|}\D\rfloor+\G\lceil \frac{n}{|w|}\D\rceil}\Prob(D_{k}^{(0)}=0)
%\]
%goes to $0$ as $l$ goes to infinity for any $0\leq n\leq+\infty$. In particular
\[
\sum_{l\geq1}\Prob(D_{l}^{(0)}=0)<+\infty.
\]
\end{lemma}
\begin{proof}
For the clarity of the presentation, let us consider the chain ${\bf E}^{(0)}$ which is defined using ${\bf D}^{(0)}$ as follows. $E_{i}^{(0)}=0$ for $i\leq0$, and for $i\geq1$
\begin{equation}
E^{(0)}_{i}=\left\{
\begin{array}{ccc}
D^{(0)}_{i}&\textrm{whenever }&D_{i}^{(0)}=i-i^{(0)}\,\,\textrm{or}\,\,0\\
E_{i-1}^{(0)}&\textrm{otherwise}.
\end{array}
\right.
\end{equation}
The behavior of ${\bf E}^{(0)}$ is easier to understand than ${\bf
  D}^{(0)}$ and their relationship is illustrated in Figure \ref{fig:combine2} for given
samples $\bar{\bf Z}$ and $\bar{\bf L}$. At time $j>0$, supposing that
$E^{(0)}_{j-1}=n\geq0$,
\begin{equation}\label{31}
E_{j}^{(0)} = \left\{
\begin{array}{lll}
j-j^{(0)} &\textrm{if $U_{j|w|-i+1}\in I(w_{-l})$ for any $i=0,\ldots,|w|-1$, (i.e. $\bar{Z}_{j}=1$)}\\
n&\textrm{if $\alpha_{-1}\leq U_{j|w|-i+1}<\alpha_{n-1}$ for any $i=0,\ldots,|w|-1$}\\
0&\textrm{if $ U_{j|w|-i+1}\geq\alpha_{n-1}$ for some $i=0,\ldots,|w|-1$.}
\end{array}
\right.
\end{equation}

It is clear  that $\Prob(D_{l}^{(0)}=0)=\Prob(E_{l}^{(0)}=0)$ and that the state $0$ is renewal for ${\bf E}^{(0)}$. It follows from \citet[][Chapter XIII.10, Theorem 1]{feller/1968} that $\mathbb{P}(E_{k}^{(0)}=0)$ is summable in $k$ if and only if the state $0$ is transient. Denote by $\zeta$ the first time after time $0$ that the chain ${\bf E}^{(0)}$ returns to the state $0$, and for $k\geq1$ we put $f_{k}=\mathbb{P}(\zeta=k)$. 
\begin{figure}[htp]
\centering
\includegraphics[scale=0.9]{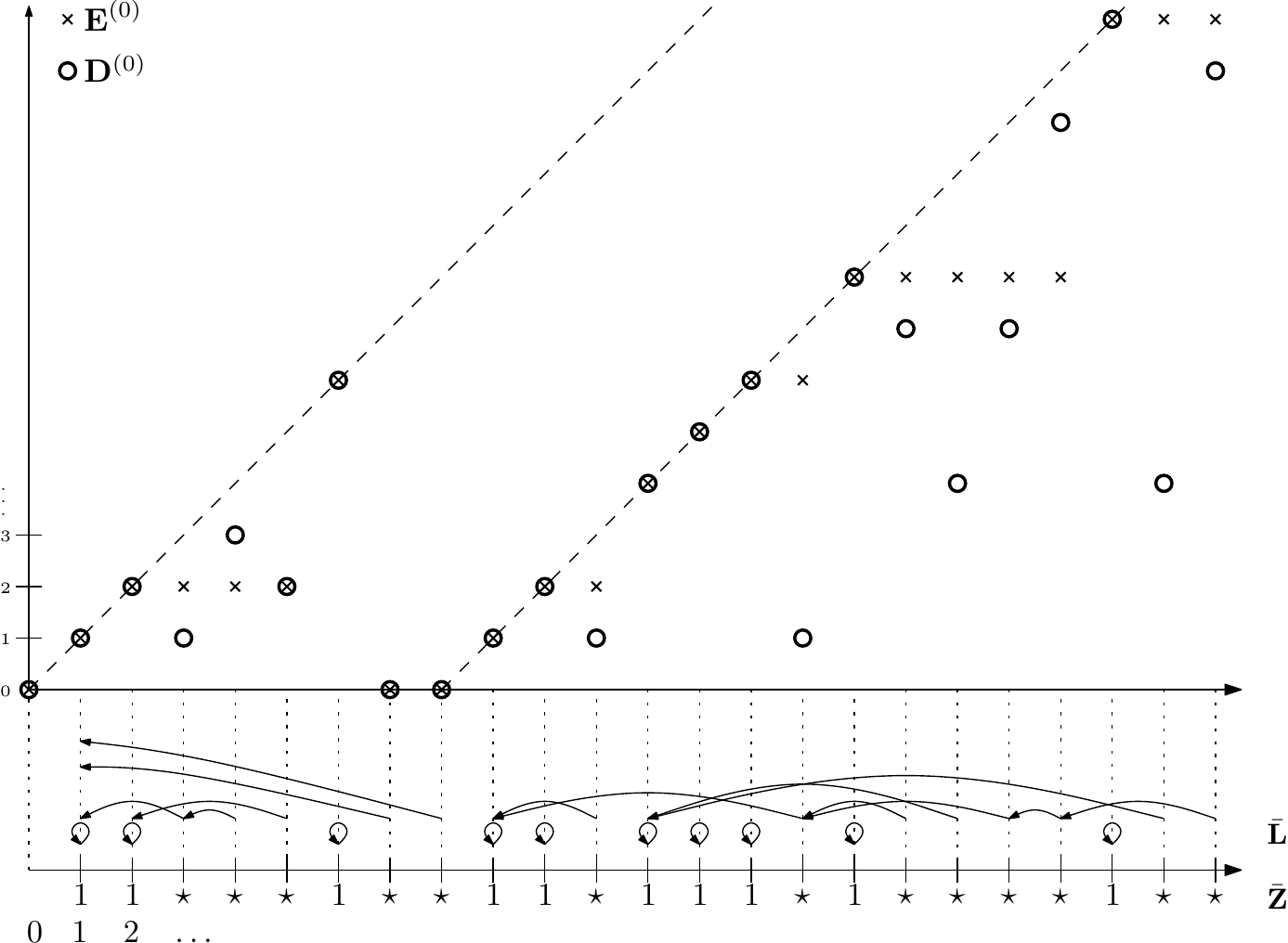}
\caption{Figure illustrating the behavior of the chains ${\bf E}^{(0)}$ and ${\bf D}^{(0)}$ together, both using the samples of $\bar{\bf Z}$ and $\bar{\bf L}$.}
\label{fig:combine2}
\end{figure}

We want to show that under the assumption of Theorem \ref{theo2}, the
state $0$ is transient for ${\bf E}^{(0)}$. Let us denote by $G_{i,l}$
the event $\{\zeta=i\}\cap\{\bar{Z}_{1}^{l}=1^{l}\}$ for $l\geq1$ and
$i\geq l+1$. If $\bar{Z}_{1}^{M}=1^{M}$ for some $M\geq1$, then $E_{M}^{(0)}=M$, $G_{i,M}=\emptyset$ for $i\leq M$ and for any $i\geq M+1$
\[
G_{i,M}=\{\bar{Z}_{1}^{M}=1^{M}\}\cap\{E_{j}^{(0)}\geq M:j=M+1,\ldots,i-2\}\cap\bigcup_{k=M}^{i-1}\{E_{i-1}^{(0)}=k\}
\cap\{E_{i}^{(0)}=0\}.
\]
The definition of ${\bf E}^{(0)}$ implies that whenever $\{\zeta=i\}$ for $i\geq1$, 
\[
\{E_{i-1}^{(0)}=k\}=\{E_{j}^{(0)}=k\,\textrm{ for }\, j=k,\ldots,i-1\}
\]
for $1\leq k\leq i-1$. It follows that
\[
G_{i,M}=\{\bar{Z}_{1}^{M}=1^{M}\}\cap\bigcup_{k=M}^{i-1}\{E_{j}^{(0)}\geq M:j=M+1,\ldots,k-1\}
\]
\[
\cap\{E_{j}^{(0)}=k:j=k,\ldots,i-1\} \, \cap \, \{E_{i}^{(0)}=0\}.
\]
Using (\ref{31}) and the chain ${\bf U}$, one obtains that for $m\geq1$ and $i\geq M+1$
\begin{equation}\label{equation}
G_{i,M}=\bigcup_{k=M}^{i-1}\left\{\begin{array}{c}\underbrace{\left\{\bar{Z}_{1}^{M}=1^{M}\right\}\cap\left\{E_{j}^{(0)}\geq M:j=M+1,\ldots,k-1\right\}}_{\textrm{event}\,\,\mathcal{B}}\\
\cap\underbrace{\{\bar{Z}_{k}=1\}\cap\bigcap_{l=k+1}^{i-1}\bigcap_{m=(l-1)|w|+1}^{l|w|}\{\alpha_{-1}\leq U_{m}<\alpha_{(k-1)|w|}\}}_{\textrm{event}\,\,\mathcal{C}}\\
\cap\underbrace{\bigcup_{j=(i-1)|w|+1}^{i|w|}\{U_{j}\geq\alpha_{(k-1)|w|}\}}_{\textrm{event}\,\,\mathcal{D}}
\end{array}\right\}
\end{equation}
where the event $\mathcal{B}$ is $\mathcal{F}(U_{1}^{(k-1)|w|})$-measurable, the event $\mathcal{C}$ is $\mathcal{F}(U_{(k-1)|w|+1}^{(i-1)|w|})$-measurable and the event $\mathcal{D}$ is $\mathcal{F}(U_{j=(i-1)|w|+1}^{i|w|})$-measurable. Therefore, they are independents. Recall that $w\in\mathcal{E}^{\star}$ and assume that 
\[
\inf_{i=1,\ldots,|w|}\inf_{\underline{z}}P(w_{-i}|\underline{z})=\epsilon>0.
\]
Using the partition
\[
\bigcup_{i\geq1}\{\zeta=i\}=\left(\bigcup_{i\geq1}\{\zeta=i\}\cap\{\bar Z_{1}^{M}=1^{M}\}\right)\cup \left(\bigcup_{i\geq1}\{\zeta=i\}\cap\{\bar Z_{1}^{M}\neq1^{M}\}\right),\,\,\,\,\forall M\geq1
\]
one obtains the following upper bound (recall that $G_{i,M}=\emptyset$ for $i\leq M$):
\[
\sum_{i\geq1}f_{i}\leq\sum_{i\geq M+1}\mathbb{P}(G_{i,M})+(1-\epsilon^{|w|M}),
\]
which holds \emph{for any } $M\geq1$.
Using the fact that $\alpha_{k}\leq1$ for any $k$, we have
\begin{eqnarray*}
\mathbb{P}(\mathcal{B}) &\le& \epsilon^{|w|M}, \\
\mathbb{P}(\hspace{0.5mm}\mathcal{C}) & = & (\alpha_{|w|(k-1)} -
\alpha_{-1})^{|w|(i-k-1)} \le  (1 -\alpha_{-1})^{|w|(i-k-1)}\\
\mathbb{P}(\mathcal{D}) &\le& |w| (1 - \alpha_{|w|(k-1)}).
\end{eqnarray*}

Therefore,  equality
(\ref{equation}) gives us the following upper bound for any  $i\geq M+1$
\[
\mathbb{P}(G_{i,M})\leq |w|\epsilon^{|w|M}\sum_{k=M}^{i-1}(1-\alpha_{-1})^{|w|(i-k-1)}(1-\alpha_{(k-1)|w|}).
\]
We have 
%Let us show that under the conditions of Theorem \ref{theo2}, 
%\[
%\sum_{i\geq1}\sum_{k=1}^{i-1}(\alpha_{(k-1)|w|}-\alpha_{-1})^{(i-k-1)|w|}(1-\alpha_{(k-1)|w|})<+\infty.
%\]
\[
\sum_{i\geq M+1}\sum_{k=M}^{i-1}(1-\alpha_{-1})^{|w|(i-k-1)}(1-\alpha_{(k-1)|w|})
=\sum_{k\geq M}(1-\alpha_{(k-1)|w|})\sum_{i\geq k+1}(1-\alpha_{-1})^{(i-k-1)|w|}
\]
where we interchanged  the order of the sums, this last equation yields
\begin{equation}\label{enfin}
\sum_{i\geq1}f_{i}\leq\frac{|w|\epsilon^{|w|M}}{1-(1-\alpha_{-1})^{|w|}}\sum_{k\geq M}(1-\alpha_{(k-1)|w|})+(1-\epsilon^{|w|M}).
\end{equation}
Under the conditions of Theorem \ref{theo2}, $\sum_{k\geq M}(1-\alpha_{(k-1)|w|})$ goes to $0$ as $M$ increases. This means that the right hand side of (\ref{enfin}) is strictly smaller than $1$ for some sufficiently large $M$, and it follows that $\sum_{i\geq1}f_{i}<1$. This finishes the proof of Lemma \ref{aah}.
\end{proof}

%\begin{proof}[Proof of Theorem \ref{theo2}]
%%By stationarity, we have for any $-\infty<m\leq n\leq+\infty$
%%\[
%%\mathbb{P}(\theta[m,n]>-\infty)=\mathbb{P}(\theta[0,n-m]<-l),
%%\]
%%and u
%Follows directly from inequality (\ref{yo}) and Lemma \ref{aah}.
%\end{proof}
\section{Conclusion}

\cite{comets/fernandez/ferrari/2002} use the uniform continuity
assumption $\alpha_k^{CFF} \rightarrow 1$. Perfect simulation under a
weaker condition was done recently by \cite{desantis/piccioni/2010}
requiring only punctual continuity, ie, $\alpha^{CFF}_k(\underline{a})
\rightarrow 1$ for any $\underline{a}$ in the set of ``admissible
histories'' (see Section 2 therein). Our extension allows to consider
kernels $P$ having discontinuities along all the points
$\underline{a}\in A^{-\infty}(\bar{w})$, for any
$w\in\mathcal{E}^{\star}$, and \emph{a priori}, no assumption is made
on the set of ``admissible histories'', so that it is generically the
set $A^{-\mathbb{N}}$. 
More specifically, consider a transition probability kernel $P$ such
that $\alpha_{k}^{CFF}$ satisfies 
$\sum_{k\geq0}(1-\alpha^{CFF}_{k})<+\infty$. It follows that, for any
$w\in\mathcal{E}^{\star}$,
\[
\sum_{k\geq0}(1-\alpha^{w}_{k})<+\infty.
\]
Now consider any $\tilde{P}$ satisfying that
$\tilde{\alpha}^{w}_{k}=\alpha^{w}_{k}$ and allowing discontinuities
along branches not containing the string $w$. Theorem \ref{theo2} says
that we still can make a perfect simulation of the unique stationary
chain compatible with $\tilde{P}$. This shows that our result is a
strict generalization of the work of
\cite{comets/fernandez/ferrari/2002} whenever we are in the regime
$\sum_{k\geq0}(1-\alpha^{CFF}_{k})<+\infty$. 

Also, our condition does not necessarily fit under the conditions of
\cite{desantis/piccioni/2010}.  It is possible to see this checking
Equation (35) of Example 1 in their work. Their notation corresponds
to
$$a_0(-1) = \alpha(-1), \quad a_0(1) = \alpha(1) \quad \mbox{ and }
a_{\infty} = \lim_{k \rightarrow \infty} \alpha^{CFF}_{k}.$$

Taking $w=(-1)(1)$ or $w=(1)(-1)$, Theorem \ref{theo2} says that we
only need $a_0(-1)$ and $a_0(1)$ to be strictly positive without any
assumption on $a_{\infty}$. We also mention that this particular example 
was already handled by the results of \cite{gallo/2009}. \\

Let us finish with some questions. The condition
$\sum_{k\geq0}(1-\alpha^{w}_{k}) <+\infty$ guarantees that the perfect
simulation scheme stops at a finite time $\theta$ which has finite
expected value. Can we find weaker conditions such that $\theta$ is
finite a.s. but has infinite expectation? Is the minorization
assumption on our reference string ($w\in\mathcal{E}^{\star}$)
necessary to obtain a practical coupling from the past algorithm for
our class of (non-necessarily continuous) chains of infinite memory?

\bibliographystyle{jtbnew}
%\bibliography{/Users/naviana/Documents/sandro/etudes/paper/artigo.simulacao.perfeita/sandro_bibli}
%\bibliography{/Users/sandro/Documents/etudes/bibtex/sandro_bibli}
\bibliography{/Users/sandro/Dropbox/Sandro/etudes/bibtex/sandro_bibli}
%\bibliography{sandro_bibli}

\vspace{2cm}
 
Sandro Gallo, Nancy L. Garcia

Instituto de Matem\'atica, Estat\'{\i}stica e Computa\c c\~ao
Cient\'\i fica 

Universidade Estadual de Campinas

Rua Sergio Buarque de Holanda, 651

13083-859 Campinas, Brasil

e-mail: {\tt nancy@ime.unicamp.br}, {\tt gsandro@ime.unicamp.br}

\end{document}